\documentclass{amsart}
\usepackage{latexsym,amsxtra,amscd,ifthen}
\usepackage{amsfonts}
\usepackage{verbatim}
\usepackage{amsmath}
\usepackage{amsthm}
\usepackage{amssymb}
\usepackage{color}
\usepackage{xcolor}
\usepackage{hyperref}
\usepackage{cancel}
\usepackage{dsfont}
\usepackage{array}
\usepackage{color}
\usepackage{bbold}
\usepackage{multirow}
\usepackage{bbm}
\usepackage{graphicx}
\usepackage{float}
\hypersetup{
 colorlinks  = true,  
 urlcolor   = black,  
 linkcolor  = black,  
 citecolor  = black   
}
\usepackage{longtable}
\definecolor{green}{RGB}{20,140,10}

\numberwithin{equation}{section}

\mathchardef\mhyphen="2D

\theoremstyle{plain}
\newtheorem{theorem}{Theorem}[section]
\newtheorem*{theorem*}{Theorem}
\newtheorem{lemma}[theorem]{Lemma}
\newtheorem{proposition}[theorem]{Proposition}

\newtheorem{corollary}[theorem]{Corollary}

\theoremstyle{definition}
\newtheorem{definition}[theorem]{Definition}
\newtheorem{example}[theorem]{Example}

\newtheorem{notation}[theorem]{Notation}
\newtheorem{remark}[theorem]{Remark}
\newtheorem{question}[theorem]{Question}

\makeatletter              
\let\c@equation\c@theorem  
\makeatother

\renewcommand{\binom}[2]{\ensuremath{\genfrac(){0pt}{2}{#1}{#2}}}

\newcommand{\Span}{\ensuremath{\operatorname{span}}}

\DeclareMathOperator{\hdet}{\operatorname{hdet}} 
\DeclareMathOperator{\GL}{GL}

\DeclareMathOperator{\Ext}{Ext}

\DeclareMathOperator{\GKdim}{GKdim}
\DeclareMathOperator{\End}{End}

 \DeclareMathOperator{\Hom}{Hom}

\newcommand{\be}{\begin{enumerate}}
\newcommand{\ee}{\end{enumerate}}
\newcommand{\bq}{\begin{eqnarray*}}
\newcommand{\eq}{\end{eqnarray*}}
\newcommand{\bqn}{\begin{eqnarray}}
\newcommand{\eqn}{\end{eqnarray}}

\newcommand{\bb}{\textnormal{b}}

\renewcommand{\i}{\mathbbm{i}}

\newcommand{\constb}{u_1}
\newcommand{\constc}{u_2}
\newcommand{\constd}{u_3}
\newcommand{\conste}{u_4}

\renewcommand{\aa}{\mathbf{a}}
\renewcommand{\bb}{\mathbf{b}}
\newcommand{\xx}{\mathbf{x}}

\usepackage{subcaption}

\begin{document}
\title[Actions of Drinfeld doubles]{Actions of Drinfeld doubles of finite groups on Artin-Schelter regular algebras}

\author{Ellen Kirkman, W. Frank Moore, Tolulope Oke}

\address{Kirkman: Department of Mathematics ,
P. O. Box 7388, Wake Forest University, Winston--Salem, NC 27109}
\email{kirkman@wfu.edu}

\address{Moore: Department of Mathematics,
P. O. Box 7388, Wake Forest University, Winston--Salem, NC 27109}
\email{moorewf@wfu.edu}

\address{Oke: Department of Mathematics,
P. O. Box 7388, Wake Forest University, Winston--Salem, NC 27109}
\email{oket@wfu.edu}
\begin{abstract} 
For a finite group $G$ and $\Bbbk$ an algebraically closed field of 
characteristic zero we consider Artin-Schelter regular algebras
$A$ on which the Drinfeld
double $D(G)$ acts inner faithfully, and its associated algebras of
invariants $A^{D(G)}$. Explicit computations for the cases when $G$ is the (generalized) quaternion group of order 8 and 16 are given.
\end{abstract}
\maketitle

Classical invariant theory concerns the actions of groups $G$ on commutative polynomial rings
$A=\Bbbk[x_1, \dots, x_n]$ for $\Bbbk$ a field, and studies the subring of invariants $A^G$.
There has been recent work exploring extensions of these results to actions on noncommutative
algebras $A$, particularly to Artin-Schelter regular algebras (Definition \ref{defASreg}). AS regular algebras are seen as a natural class of algebras to consider because when they are commutative, they are isomorphic to commutative polynomial algebras.  Many results in classical invariant theory have been shown to have analogous results in this context (see, for example, \cite{K}). One can begin by considering actions by subgroups of the graded automorphism group of $A$,
 i.e. study symmetries of $A$.  However, given the
rigidity of noncommutative algebras, there often are not a large number of such actions.
Hence to find additional actions on $A$ there has been interest in considering ``quantum symmetries" of noncommutative algebras,
i.e. actions on $A$ by Hopf algebras $H$.  However, not all algebras have non-trivial quantum symmetries (i.e. Hopf actions that
are not actually actions by a group); see, for example, results of  Cuadra, Etingof, and Walton, including \cite{EW1, EW2, EW3, CEW1, CEW2}. The search for pairs $(A,H)$, where $A$ is an Artin-Schelter regular algebra and $H$ is a finite dimensional Hopf algebra that acts inner faithfully on $A$ (meaning there is no proper Hopf factor of $H$ that acts non-trivially on $A$) has produced a number of examples. For
Artin-Schelter regular algebras that have a grading by a finite group $G$ there are actions by $H = (\Bbbk G)^*$ the dual of the group algebra, which when $G$ is non-abelian is
not isomorphic to $\Bbbk G$ (see, e.g.  \cite{KKZ2, CKZ1, CKZ2, C2, Zhu}) and  pairs $(A,H)$ for other Hopf algebras  have also been discovered (see, e.g. \cite{FKMW1,FKMW2}).  At this time there is not a good understanding of the Hopf algebras $H$ or the AS regular algebras $A$ for which a pair $(A,H)$ exists.

Throughout, as we choose specific values of some of the parameters, we take $\Bbbk = \mathbb{C}$, though the results could be considered over an algebraically closed field of characteristic zero. 
In this paper, with the goal of creating more examples of pairs $(A,H)$, we consider the Drinfeld double $H = D(G)$ of a finite group, which is constructed using
both $\Bbbk G$ and $(\Bbbk G)^*$. After experimenting with a number of smaller order nonabelian groups, we found that $D(G)$ has inner faithful actions on AS regular algebras when $G$ is the (generalized) quaternion group (dicyclic group) $\mathfrak{D}_{2n}$ of order $8\; (n=2)$ and of order $16\; (n=4)$, and we consider these cases in detail.  Furthermore, we show that there is a large number of
algebras on which $\mathfrak{D}_{2n}$, for $n = 2$ and $4$, acts inner faithfully, though in many cases the algebra $A$ on which $H$ acts will not be an AS regular algebra.

When the characteristic of $\Bbbk$ is zero, $H = D(G)$ is a semisimple Hopf algebra. Using results of Witherspoon \cite{W1, W2},  we can construct the simple $D(G)$-modules. Then using
a result of Rieffel \cite{R} (see also \cite{FKMW1}), we can consider all possible finitely generated $D(G)$-modules $V$, expressed as a finite sum of simple $D(G)$-modules, and search for $V$ where the action of $D(G)$ on $\sum_{m=1}^{\infty}V^{\otimes m}$ contains all simple $D(G)$-modules as direct summands. The $D(G)$-modules $V$ with this property can serve as the $D(G)$-module of the degree one elements of an algebra $B$ (i.e. $B_1 = V$ as $D(G)$-modules) with $D(G)$ acting inner faithfully on $B$.  We choose such a $D(G)$-module $V$ that is the sum of a minimal number of simple $D(G)$-modules.  These computations were performed with GAP \cite{GAP} and Macaulay2 \cite{M2}, and for $D(\mathfrak{D}_4)$ there are 224 such modules and for $D(\mathfrak{D}_{8})$ there are 1,952 of them! These modules $V$ can then used to define algebras $B$ with $B_1 = V$ that support inner faithful $D(G)$ actions (i.e. provide the $D(G)$-module structure for the degree one elements $B_1$ of an algebra $B$, where the $D(G)$ action on $B$ is inner faithful).

Next, having fixed the action of $D(G)$ on the generators of $B$, we need to define the relations in $B$. In order to define a quadratic algebra $B$ on which $D(G)$ acts, the ideal of relations for $B$ must be an $D(G)$-submodule of $V \otimes V$; to accomplish this, we explicitly compute the generators of the simple modules in the decomposition of $V \otimes V$  into simple $D(G)$-modules, producing a generator for each simple module in the decomposition. We then seek combinations of these generators that produce algebras $B$ that are $D(G)$-modules and AS regular algebras.  In both cases, $D(\mathfrak{D}_4)$ and $D(\mathfrak{D}_{8})$, while not exhausting all possible relations for the algebras $B$, we find algebras $B$ that are (trimmed) double Ore extensions in the sense  of J. J. Zhang and J. Zhang \cite{ZZ} (see Notation \ref{doubleOreNotation}), and hence are AS regular. In Theorem \ref{double-ore} for the given values of the six parameters, the six-dimensional algebras $B$ for $D(\mathfrak{D}_{4})$ are shown  to be double Ore extension of a skew-polynomial ring, and in Theorem \ref{i-double-ore} for the given values of the seven parameters, the eight-dimensional algebras $B$ for $H = D(\mathfrak{D}_{8})$ are shown  to be iterated double Ore extensions of a $(-1)$-skew polynomial ring. 

Using a result of Crawford \cite{C1}, we compute the homological determinant of the actions. Furthermore, in the case $n=2$, using a degree bound argument of \cite[Lemma 2.2]{KKZ1}, and a noncommutative analogue of SAGBI bases in Lemma \ref{sagbi-bases}, we find a minimal generating set for the invariant subring $B^H$, when $H = D(\mathfrak{D}_4)$, under the condition that the homological determinant of the action of $D(\mathfrak{D}_4)$ is trivial, and present our result in Theorem \ref{generatorTheorem}. 

 The dicyclic groups $\mathfrak{D}_{2n}$ for $n = 2$ and $n=4$ were found to have AS regular algebras associated to their doubles, but we show that pairs $(A,D(G))$ do not exist for all finite groups $G$. In particular, we show that finding a quadratic AS regular algebra $A$ with a  inner faithful $D(S_3)$ action  using the smallest number of generators is not possible.  Understanding the groups $G$ for which  there is an AS regular algebra $A$ supporting an inner faithful $D(G)$ action remains an open problem, and properties of all the rings of invariants remain to be explored. However, these examples suggest that this context is rich for further study.

The paper is organized as follows.  Section \ref{background} provides background on the Drinfeld double of a finite group and its invariants, as well as on double Ore extensions, and the dicyclic groups $\mathfrak{D}_{2n}$ are introduced. In Section \ref{dicyclic-n-2} we begin the study of the case $n=2$ where $G$ is the quaternion group of order $8$.
 In Proposition \ref{inner-faithful-n-2} we provide a particular $D(\mathfrak{D}_{4})$-module $V$  that supports an inner faithful action for algebras $B$ with relations as given in Table \ref{relationsQ4}.  
 In Theorem \ref{double-ore} we explicitly present the AS regular algebras $B$ with their generators and relations, along with the action of $D(\mathfrak{D}_{4})$ on $B$, and we prove that these algebras $B$ are double Ore extensions with the specified data.  The homological determinant of the action of $D(\mathfrak{D}_{4})$ on $B$ is computed in Proposition \ref{prop:hdetQ4}.  In Section \ref{invariants} the invariants $B^{D(\mathfrak{D}_{4})}$ are computed, assuming that the homological determinant of the action is trivial.  The case 
 when $n=4$ (i.e., when $G$ is the dicyclic group of order 16) is considered in Section \ref{dicyclic-n-4}. In Section \ref{S3}, we study the action of $D(S_3)$, the double of the symmetric group of order 6. We conclude the paper in Section \ref{questions} with some remarks and open questions.

\begin{section}{Background}\label{background}
In this section, we provide notations and background results that we use throughout. We use the standard notation $(\Delta,\epsilon, S)$ for the comultiplication, counit and antipode of a Hopf algebra $H$ and refer to \cite{M} for preliminary results about Hopf algebras. All $H$-modules are right $H$-modules, and actions by $H$ are on the right (when described by matrices using the rows of the matrix).

\begin{definition} \label{defASreg} Let $A$ be a connected graded algebra. Then $A$ is \textit{Artin–Schelter (AS) regular (resp. AS Gorenstein)} if
\begin{enumerate}
\item[(1)] $A$ has finite global dimension (resp. injdim($_AA$) and injdim($A_A$ are both finite),
\item[(2)] $A$ has finite Gelfand-Kirillov dimension, and
\item[(3)] the Gorenstein condition holds i.e. $\Ext^i_{A}(\Bbbk, A) = \delta_{i,d}\cdot \Bbbk(l)$ for some $l\in\mathbb{Z}$.
\end{enumerate}
\end{definition}
Examples of AS regular algebras include skew polynomial rings, graded Ore extensions of AS regular algebras, double Ore extensions of AS regular algebra as well as other algebras described by generators and relations. We now give the following notations for skew polynomial rings and trimmed double Ore extensions that we use. 

\begin{notation} \label{skewPolynomial}
Let $\mathbf{q} = (q_{ij})$ be an $n \times n$ multiplicatively skew symmetric matrix with
entries in a field $\Bbbk$.  That is, for each $1 \leq i,j \leq n$, one has that
$q_{ij} = q_{ji}^{-1}$.    The skew-polynomial ring
$\Bbbk_\mathbf{q}[x_1,\dots,x_n]$ is defined to be the quotient of the free algebra
on $x_1,\dots,x_n$ modulo the ideal generated by the relations
$$\{x_jx_i = q_{ji}x_ix_j~|~1 \leq i,j \leq n\}.$$
\end{notation}

\begin{notation} \label{doubleOreNotation}
Let $A$ be an algebra and let $B$ be another algebra containing
$A$ as a subring.  Following \cite{ZZ}, one says that $B$
is a \emph{trimmed double Ore extension of $A$} (which we will refer to as simply a double Ore extension), provided:
\begin{enumerate}
\item There exist $x_1,x_2 \in B$ such that $B$ is generated by
$A$, $x_1$ and $x_2$;
\item One has an equality $x_1x_2 = p_{12}x_2x_1 + p_{11}x_1^2$;
\item $B$ is a free left $A$-module with basis $\{x_1^{n_1}x_2^{n_2}~|~n_i \in \mathbb{N}\}$;
\item One has an equality $x_1A + x_2A = Ax_1 + Ax_2$.
\end{enumerate}

One may also describe such an extension ``externally''.  In this case, one considers
a $k$-algebra map $\sigma : A \to M_2(A)$ and scalars
$P = \{p_{11},p_{12}\}$ such
that certain conditions are met, which we explain now in our case of
interest, which is when $p_{11} = 0$.  We will use $p$ in place of $p_{12}$,
and use $p$ in place of $P$ for brevity.

Before giving the conditions, we set up some  additional notation.  For $r \in A$,
let $\sigma(r) = \begin{pmatrix}\sigma_{11}(r) & \sigma_{12}(r) \\
                                \sigma_{21}(r) & \sigma_{22}(r) \end{pmatrix}$,
where each $\sigma_{ij} \in \End_k(A)$.
The data $\{\sigma,p\}$ gives rise to the relation $x_1x_2 = px_2x_1$ as well
as relations obtained by equating the entries in the following matrix equation:
\begin{equation*}
\begin{pmatrix}x_1\\x_2\end{pmatrix}r = \begin{pmatrix}x_1r \\ x_2r\end{pmatrix}
 = \sigma(r)\begin{pmatrix}x_1\\x_2\end{pmatrix}
 = \begin{pmatrix}\sigma_{11}(r) & \sigma_{12}(r) \\ \sigma_{21}(r) & \sigma_{22}(r) \end{pmatrix}
   \begin{pmatrix}x_1\\x_2\end{pmatrix}.
 \\
 \end{equation*}

Following \cite{ZZ}, one says that a $\Bbbk$-algebra map $\sigma : A \to M_2(A)$ is 
\emph{invertible} if there exists a $\Bbbk$-algebra map $\phi : A \to M_2(A)$ such that
one has the following equality in $M_2(\End_k(A))$: 
$$
\begin{pmatrix}
\phi_{11} & \phi_{12} \\
\phi_{21} & \phi_{22}
\end{pmatrix}
\begin{pmatrix}
\sigma_{11} & \sigma_{12} \\
\sigma_{21} & \sigma_{22}
\end{pmatrix}
= 
\begin{pmatrix}
\operatorname{Id}_A & 0 \\
0 & \operatorname{Id}_A
\end{pmatrix}
=
\begin{pmatrix}
\sigma_{11} & \sigma_{12} \\
\sigma_{21} & \sigma_{22}
\end{pmatrix}
\begin{pmatrix}
\phi_{11} & \phi_{12} \\
\phi_{21} & \phi_{22}
\end{pmatrix}.
$$


 
By \cite[Proposition 1.11]{ZZ}, the data $\{\sigma,p\}$ (where $\sigma$ is an invertible algebra map in the sense described above) defines a double Ore
extension $A_p[x_1,x_2;\sigma]$ provided that the following conditions hold for all
elements $r$ in an algebra generating set of $A$ over $k$:
\begin{equation} \label{doubleOreConditions}
\begin{gathered}
\sigma_{21}(\sigma_{11}(r)) = p\sigma_{11}(\sigma_{21}(r)), \\
\sigma_{21}(\sigma_{12}(r)) - p^2\sigma_{12}(\sigma_{21}(r)) = p(\sigma_{11}(\sigma_{22}(r)) - \sigma_{22}(\sigma_{11}(r))), \\
\sigma_{22}(\sigma_{12}(r)) = p\sigma_{12}(\sigma_{22}(r)).
\end{gathered}
\end{equation}
\end{notation}

\begin{remark} \label{diagDOData}
We will  consider only graded double Ore extensions, so that one may view $\sigma_{ij}$ as
an element of $\GL(A_1)$, and hence as a matrix after choosing a basis of $A_1$.
In many cases that we consider, $\sigma(r)$ is either a diagonal or a skew-diagonal
matrix.  If $\sigma(r)$ is a diagonal (respectively, skew-diagonal) matrix for all $r$
in a generating set of $A$, then equations  \eqref{doubleOreConditions} reduce to the
condition that the matrices $\sigma_{11}(r)$ and $\sigma_{22}(r)$
(respectively, $\sigma_{12}(r)$ and $\sigma_{21}(r)$) commute for all generators $r$ of $A$.
\end{remark}

\begin{lemma} \label{MinsideMM}
Let $\Bbbk$ be a field, $R$ a semisimple $\Bbbk$-algebra, $M$ a right $R$-module, and let
$\textbf{a} = [a_1:\dots:a_n] \in \mathbb{P}^{n-1}(\Bbbk)$.
Then the function $\varphi_\textbf{a} : M \to M^{\oplus n}$ given by
$\varphi_\textbf{a}(m) = (a_1m,\dots,a_nm)$ is an injective homomorphism of
$R$-modules.  If $M$ is simple and $\Bbbk$ is algebraically closed,
then every simple submodule of $M^{\oplus n}$ is given by the
image of such a map.
\end{lemma}

\begin{proof}
The map $\varphi_\textbf{a}$ is clearly an injective homomorphism of $R$-modules.
Now suppose that $M$ is simple with $\Bbbk$ algebraically closed, 
and let $N$ be a simple submodule of $M^{\oplus n}$.  Then $N$
is isomorphic to $M$, and is a direct summand of $M^{\oplus n}$.
Fix an isomorphism $\psi : M \to N$, and for each $i = 1,\dots,n$, let
$\pi_i$ denote the projection onto the $i^\text{th}$ factor of $M^{\oplus n}$.
By Schur's Lemma, one has that for each $i$, there exists $a_i \in \Bbbk$
such that $\pi_i(\psi(m)) = a_im$.  Since $N$ is nonzero,
some $a_i \neq 0$ so that $[a_1:\dots:a_n] \in \mathbb{P}^{n-1}$.  It follows
that $N = \{(a_1m,\dots,a_nm)~:~m \in M\}$, as desired.
\end{proof}

\begin{notation}\label{Drinfeld-double}
For a finite group $G$, let $\Bbbk G$ be the group algebra and $(\Bbbk G)^*:= \Hom_\Bbbk(\Bbbk G,\Bbbk)$ be its Hopf dual. A $\Bbbk$-basis for $(\Bbbk G)^*$ is given by $\{\phi_g: g\in G\}$ where 
 $$\phi_g(h) = \begin{cases} 1 & g=h, \\ 0 & g\neq h.\end{cases}$$
The multiplication on  $(\Bbbk G)^*$ is pointwise, that is, $(\phi_g\phi_h)(x) = \phi_g(x)\phi_h(x)$ and the left action $G$ on $(\Bbbk G)^*$ is given by $h\cdot\phi_g\mapsto \phi^h_g$ where $\phi^h_g = \phi_{hgh^{-1}}$ for all $g,h\in G$. Following the definitions in \cite{W1}, the Drinfeld double of $G$ is the vector space $D(G) = \Bbbk G^* \otimes_{\Bbbk}  \Bbbk G$, having a $\Bbbk$-basis $\{\phi_g \otimes h: g,h\in G\}$ which after suppressing the tensor will be written as $\{\phi_g h: g,h\in G\}$. The multiplication
in $D(G)$ is defined on basis elements as $(\phi_g h) (\phi_{g'} h') = \phi_g (h\cdot\phi_{g'}) h h' = \phi_g \phi_{hg'h^{-1}} hh',$
which is nonzero if and only if $g = hg'h^{-1}$. The coproduct on $D(G)$ is defined as
$$\Delta(\phi_g h) = \sum_{x \in G} \phi_xh \otimes \phi_{x^{-1} g}h \in D(G) \otimes D(G),$$
the counit is $\epsilon(\phi_g h) = \delta_{1,g}$, and the antipode is $ S(\phi_g h) = \phi_{h^{-1}g^{-1}h} h^{-1}$. Thus a Hopf structure on the Drinfeld double is defined.
\end{notation}

\begin{notation} \label{not:modulesOverDouble}
Each simple $D(G)$-module (c.f. \cite{W1}[Corollary 2.3]) is of
the form $(a,\chi)$, where $a$ represents a conjugacy
class of $G$ and $\chi$ is an irreducible character 
of $C_G(a)$, the centralizer of $a$ in $G$. For 
$a,g \in G$ we use the notation $a^g = g^{-1}a g$.

The simple right $D(G)$-module  $(a,\chi)$ has a basis $v \otimes_{C_G(a)} g_i \in V \otimes _{C_G(a)} \Bbbk G$, where $g_i$ ranges over the $[G:C_G(a)]$ elements in a fixed transveral $\{ g_i\}$ for the right cosets  $C_G(a)g_i$ and $v$ ranges over a basis for the simple $C_G(a)$-module $V$ associated to $\chi$. The module  $(a, \chi)$ is graded with $(v \otimes_{C_G(a)} g_i)$ in degree $a^{g_i}$. The action of $D(G)$ on a basis element of the module $(a, \chi)$ is:
$$(v \otimes_{C_G(a)} g_i)(\phi_x h)=
 \begin{cases}
0 & x \neq a^g_i,\\
v \otimes_{C_G(a)} g_ih & x = a^{g_i}.
\end{cases}$$
Then $v \otimes_{C_G(a)} g_ih$ is expressed in terms of the basis by
writing $g_ih = cg_j$, where $g_ih$ is in the coset $C_G(a)g_j$, so
$v \otimes_{C_G(a)} g_ih = v \otimes_{C_G(a)} cg_j = v.c \otimes_{C_G(a)} g_j$.
\end{notation}


The following lemma can be used to find invariants of $A$ under the action of $D(G)$.
\begin{lemma}\label{identity-component}
Let $A$ be a right $D(G)$-module algebra.  Then:
\begin{enumerate}
\item The action of $G$ induced from that of $D(G)$ restricts to an action of $A_e$, the identity
component of $A$.  The invariants of this restricted action are the invariants of the action of $D(G)$ on $A$.
That is, $A^{D(G)} = (A_e)^G$.
\item Let $f \in A^G$.  Let $f = \sum_{g \in G} f_g$ be the decomposition of $f$ into
its $G$-graded components.  Then for every conjugacy class $X$ of $G$,
$f_X := \sum_{x \in X} f_x \in A^G$.  In particular, $f_e \in A^G$.
\end{enumerate}
\end{lemma}

\begin{proof}
Let $f \in A_e$.  In particular, we have that $\phi_e f = f$.  Then for any $g \in G$, we have that
$f.g = (f\phi_e).g = f.(\phi_e g) = f.(g\phi_e) = (f.g)\phi_e \in A_e$, so that the first claim is justified.
Next, let $f = \sum_{g \in G} f_g \in A^G$, let $X$ be a conjugacy class of $G$, and let
$f_X = \sum_{x \in X} f_x$ be as in the second claim.  Then for all $g \in G$, one has that
$$f_X.g = \left(\sum_{x\in X}f.\phi_x\right).g = \sum_{x \in X} f.(\phi_x g)
        = \sum_{x \in X} (f.g)\phi_{g^{-1}xg} = \sum_{x \in X} f.\phi_{g^{-1}xg} = f_X.$$
\end{proof}

\begin{definition}
A Hopf algebra $H$ is said to act {\it inner faithfully} on an $H$-module $V$ if there is no non-zero Hopf ideal $I$ with $V I = 0$.
\end{definition}
We make use of the following theorem to show that a $D(G)$-module is inner faithful.
\begin{theorem}\cite[Theorem 1.4]{FKMW1}
Let $V$ be a module over a finite-dimensional semisimple Hopf algebra $H$. Then the following conditions are equivalent.
\begin{enumerate}
\item[(1)] $V$ is an inner faithful $H$-module,
\item[(2)] The tensor algebra $T(V )$ is a faithful $H$-module,
\item[(3)] Every simple $H$-module appears as a direct summand of $V^{\otimes n}$ for some $n$.
\end{enumerate}
\end{theorem}

The category of simple $D(G)$-modules is a modular tensor category, and the  fusion relations for irreducible $D(G)$-modules can be found using the $S$-matrix for the category, as described in \cite{CGR} (see also \cite{W1,W2}) using group characters.

For a general modular tensor category the ``Verlinde formula"
$$N_{V_a, V_b}^{V_c} =  \sum_{V_i} \frac{S_{a, i} S_{b,i} \overline{S_{c, i}}}{S_{0, i}}$$
gives the multiplicity of the simple module $V_c$ in $V_a \otimes V_b$ from the entries in the $S$ matrix. For the Drinfeld double of a group algebra, the entry in the $S$ matrix corresponding to the simple modules $V_a = (a, \chi), V_b= (b,\chi')$ can be computed from the character table  and centralizers by the formula
$$S_{(a,\chi),(b,\chi')} = \frac{1}{|C(a)||C(b)|} \sum \overline{\chi(gbg^{-1})} \hspace*{.1in}\overline{\chi'(g^{-1}ag)}, $$
where the sum is over the set $\{g: agbg^{-1}=gbg^{-1}a\}$. These formulas were used in some of our early computations, but, as we need the generators of the summands, these formulas are not used here.

In this paper we consider the 
{ dicyclic groups $\mathfrak{D}_{2n}$}.   
  Following the notation in \cite{CGR} let $G=\mathfrak{D}_{2n}$ be the dicyclic group of order $4n$
$$\mathfrak{D}_{2n} = \langle s,r: r^{2n} = e, \; s^2= r^n, \; r s r s^{-1} = e \rangle.$$
This group is also called the ``binary dihedral" group 
 \cite{Wiki1};
 ~it is also called the ``generalized quaternion" group when $n$ is a power of $2$ \cite{Wiki2}.
The group has $n+3$ conjugacy classes with representatives
$$\{e,\; r^k \;(1 \leq k \leq n), \;s, \;sr\}.$$
It has 4 one-dimensional representations $\psi_i$ $(i = 0 \dots, 3)$ and $n-1$ two-dimensional representations $\chi_i$ ($i = 1,\dots, n-1$).  There are $2n^2+14$ simple $D(\mathfrak{D}_{2n})$-modules (also referred to as ``primaries").
The irreducible characters of $\mathfrak{D}_{2n}$ and the centralizers of the representative of each conjugacy class are given in  Table \ref{chars} from
\cite{CGR}, where $\iota = 1$ when $n$ is even and  $\iota = \i$ (primitive 4th root of 1) when $n$ is odd.

\begin{table}[H]
\caption{Character Table for $\mathfrak{D}_{2n}$}
\begin{center}
 \begin{tabular}{|c|c|c|c|c|c|}
\hline
 $\mathfrak{D}_{2n}$& $e$ &$ r^n$ & $r^k$ & $s$ & $sr$\\
\hline
\hline
$\psi_0$ & 1 & 1 & 1& 1& 1\\
\hline
$\psi_1$ & 1 & 1 & 1& $-1$& $-1$\\
\hline
$\psi_2$ & 1 &$ (-1)^n$ & $(- 1)^k$& $\iota$ & $-\iota$\\
\hline
$\psi_3$ & 1 &$ (-1)^n$ & $(- 1)^k$& $-\iota$ & $\iota$\\
\hline
$\chi_i$ & 2 &$ 2(-1)^i$ & $2\cos(\pi i k/n)$&0& 0\\
\hline
\hline
	$C_G(a)$ &$ \mathfrak{D}_{2n}$ & $ \mathfrak{D}_{2n}$& $\mathbb{Z}_{2n}$ & $\mathbb{Z}_4$ & $\mathbb{Z}_4$ \\
\hline
\end{tabular}
\end{center}
\label{chars}
\end{table}
\end{section}

\begin{section}{\texorpdfstring{$D(\mathfrak{D}_{4})$}{D(D4)}, the case \texorpdfstring{$n=2$}{n=2}: the quaternion group of order 8}
\label{dicyclic-n-2}
~\\
In this section we begin the consideration of the double of the
quaternion group.  We will use the simple $D(\mathfrak{D}_{4})$-modules to 
construct a $D(\mathfrak{D}_{4})$-module $V$ that can provide an  action of
$D(\mathfrak{D}_{4})$ on the generators of an algebra $B$ so that the action on 
$B$ is inner faithful.  There are 22 simple $D(\mathfrak{D}_{4})$-modules
$V_i$ for $i = 0, \dots, 21$, where $V_0$ is the trivial
$\mathfrak{D}_{4}$-module.

The conjugacy classes (with choice of representative $a$), 
centralizers, and representations are given in Table \ref{simpleDD4}.
Indeed, $\psi_j$ has action of $r$ and $s$ as given in
Table \ref{chars}, and $\chi$ has action given by
$u.r = v$, $v.r = -u$, $u.s = \i u$ and $v.s = \i v$.

In addition, one has that $\alpha_j$ has action given by $t.r = \i^jt$,
$\beta_j$ has action given by $t.s = \i^jt$, and $\gamma_j$ has action given
by $t.(sr) = \i^jt$, where $\i$ is a primitive fourth root of unity, and
$0 \leq j \leq 3$.
\begin{table}[H]
\caption{Simple $D(\mathfrak{D}_4)$-modules}
\begin{center}
 \begin{tabular}{|c|c|c|c|}
\hline
  Conj class & $C_G(a)$ & Rep $C_G(a)$& transversal\\
\hline
\hline
$\{a=e\}$ & $\mathfrak{D}_{4}$ & $\psi_0, \psi_1, \psi_2, \psi_3, \chi$& $\{e\}$\\
\hline
 &  & $V_0, V_1, V_2, V_3, V_4$& \\
\hline \hline
$\{a=r^2\}$ & $\mathfrak{D}_{4}$ & $\psi_0, \psi_1, \psi_2, \psi_3, \chi$& $\{e\}$\\
\hline
 &  & $V_5, V_6, V_7, V_8, V_9$ & \\
\hline \hline
$\{a=r, r^3\}$ & $\mathbb{Z}_4 = \langle r \rangle $ & $\alpha_0, \alpha_1, \alpha_2, \alpha_3$&$\{e, s\}$ \\
\hline
 &  & $V_{10}, V_{11}, V_{12}, V_{13}$ & \\
\hline \hline
$\{a=s, sr^2\}$ & $\mathbb{Z}_4= \langle s \rangle $ & $\beta_0, \beta_1, \beta_2, \beta_3$ & $\{e, r\}$\\
\hline
&  & $V_{14}, V_{15}, V_{16}, V_{17}$ & \\
\hline \hline
$\{a=sr, sr^3\}$ & $\mathbb{Z}_4 = \langle sr \rangle $ & $\gamma_0, \gamma_1, \gamma_2, \gamma_3$&$\{e, s\}$ \\
\hline
&  & $V_{18}, V_{19}, V_{20}, V_{21}$ & \\
\hline
\end{tabular}
\end{center}
\end{table}
\label{simpleDD4}

In Table \ref{irrepsD4} we give the explicit choices 
of representations of $D(\mathfrak{D}_{4})$ that we 
will use in our analysis.  Recall that by Notation 
\ref{not:modulesOverDouble}, each simple
$D(\mathfrak{D}_4)$-module arises as the
$\mathfrak{D}_4$-module induced by an irreducible
representation of the centralizer of an element of
$\mathfrak{D}_4$, with a suitable grading arising
from the conjugation action.

The columns marked
$r$ and $s$ indicate the action of $r$ and $s$ on the basis elements, which are in the column labeled ``Gen".  The group grade of the generator is given in the column labeled ``Gr".
\begin{table}[H]
\caption{Explicit action of $D(\mathfrak{D}_4)$ on its simple modules}
\begin{minipage}[t]{.45\textwidth}
\strut\vspace*{-\baselineskip}\newline
\begin{center}
\begin{tabular}{|c|c|c|c|c|c|} \hline
Class & $V_i$ & Gen. & Gr. & $r$ & $s$ \\ \hline\hline
{\multirow{6}{*}{$[e]$}}
 & $V_0$ & $u$ & $e$ & $u$ & $u$ \\ \cline{2-6}
 &  $V_1$ & $u$ & $e$ & $u$ & $-u$ \\ \cline{2-6}
 &  $V_2$ & $u$ & $e$ & $-u$ & $u$ \\ \cline{2-6}
 &  $V_3$ & $u$ & $e$ & $-u$ & $-u$ \\ \cline{2-6}
 &  {\multirow{2}{*}{$V_4$}} 
       &  $u$ & $e$ & $v$ & $\i u$ \\
      && $v$ & $e$ & $-u$ & $-\i v$ \\ \hline\hline
{\multirow{6}{*}{$[r^2]$}} &
     $V_5$ & $u$ & $r^2$ & $u$ & $u$ \\ \cline{2-6}
 &  $V_6$ & $u$ & $r^2$ & $u$ & $-u$ \\ \cline{2-6}
 &  $V_7$ & $u$ & $r^2$ & $-u$ & $u$ \\ \cline{2-6}
 &  $V_8$ & $u$ & $r^2$ & $-u$ & $-u$ \\ \cline{2-6}
 &  {\multirow{2}{*}{$V_9$}} 
      &  $u$ & $r^2$ & $v$ & $\i u$ \\
     && $v$ & $r^2$ & $-u$ & $-\i v$ \\ \hline\hline
{\multirow{8}{*}{$[r]$}} &
     {\multirow{2}{*}{$V_{10}$}}
      &  $u$ & $r$   & $u$ & $v$ \\ 
     && $v$ & $r^3$ & $v$ & $u$ \\ \cline{2-6}
 &  {\multirow{2}{*}{$V_{11}$}}
      &  $u$ & $r$   & $\i u$  & $v$ \\ 
     && $v$ & $r^3$ & $-\i v$ & $-u$ \\ \cline{2-6}
 &  {\multirow{2}{*}{$V_{12}$}}
      &  $u$ & $r$   & $-u$ & $v$ \\ 
     && $v$ & $r^3$ & $-v$ & $u$ \\ \cline{2-6}
 &  {\multirow{2}{*}{$V_{13}$}}
      &  $u$ & $r$   & $-\i u$ & $v$ \\ 
     && $v$ & $r^3$ & $\i v$ & $-u$ \\ \hline
\end{tabular}
\end{center}
\end{minipage}
\:\:
\begin{minipage}[t]{.45\textwidth}
\strut\vspace*{-\baselineskip}\newline
\begin{center}
\begin{tabular}{|c|c|c|c|c|c|} \hline
Class & $V_i$ & Gen. & Gr. & $r$ & $s$ \\ \hline\hline
{\multirow{8}{*}{$[s]$}} &
     {\multirow{2}{*}{$V_{14}$}}
      &  $u$ & $s$    & $v$ & $u$ \\ 
     && $v$ & $sr^2$ & $u$ & $v$ \\ \cline{2-6}
 &  {\multirow{2}{*}{$V_{15}$}}
      &  $u$ & $s$    & $v$ & $\i u$ \\ 
     && $v$ & $sr^2$ & $-u$ & $-\i v$\\ \cline{2-6}
     &  {\multirow{2}{*}{$V_{16}$}}
      &  $u$ & $s$    & $  v$ & $-u $ \\ 
     && $v$ & $sr^2$ &$ u$ & $- v$ \\ \cline{2-6}
 &  {\multirow{2}{*}{$V_{17}$}}
      &  $u$ & $s$    & $v$ & $-\i u$ \\ 
     && $v$ & $sr^2$ & $-u$ & $\i v$ \\ \hline\hline
{\multirow{8}{*}{$[sr]$}} &
     {\multirow{2}{*}{$V_{18}$}}
      &  $u$ & $sr$   & $v$ & $v$ \\ 
     && $v$ & $sr^3$ & $u$ & $u$ \\ \cline{2-6}
 &  {\multirow{2}{*}{$V_{19}$}}
      &  $u$ & $sr$   & $\i v$ & $v$ \\ 
     && $v$ & $sr^3$ & $\i u$ & $-u$ \\ \cline{2-6}
 &  {\multirow{2}{*}{$V_{20}$}}
      &  $u$ & $sr$   & $-v$ & $v$ \\ 
     && $v$ & $sr^3$ & $-u$ & $u$ \\ \cline{2-6}
 &  {\multirow{2}{*}{$V_{21}$}}
      &  $u$ & $sr$   & $-\i v$ & $v$ \\ 
     && $v$ & $sr^3$ & $-\i u$ & $-u$ \\ \hline
\end{tabular}
\end{center}
\end{minipage}
\label{irrepsD4}
\end{table}

Consider $V_{17} = (s, \beta_3), V_{20} = (sr, \gamma_2), V_{21}=(sr, \gamma_3)$.
A computation (some of the details of which are given below) shows that for this choice of 
$V = V_{17} \oplus V_{20} \oplus V_{21}$, all irreducible $D(\mathfrak{D}_{4})$-modules 
occur as direct summands of $V, V^{\otimes 2},$ $V^{\otimes 3}$ or
$V^{\otimes 4}$; see Proposition \ref{inner-faithful-n-2}.

In what follows, we decompose $V \otimes V$ for
$V = V_{17} \oplus V_{20} \oplus V_{21}$,
where our choice of basis elements for the simple modules are: $V_{17} = \Span_{\Bbbk}\{x_1,x_2\}$, $V_{19} = \Span_{\Bbbk}\{y_1,y_2\}$ and
$V_{20} = \Span_{\Bbbk}\{z_1,z_2\}$.  To compute $V \otimes V$ we first compute  the 9 possible tensor products of the simple $D(\mathfrak{D}_{4})$-modules in the decomposition of $V$. The tensor product symbols are suppressed in the description of the 
generators, and the generators are listed in the same order ($u,v$) that they
appear in Table \ref{tensorSquareD4}. The column $V_i$ gives the direct summands that occur in tensor product of the indicated simple $D(\mathfrak{D}_{4})$-modules.
\begin{table}[H]
\caption{Decomposition of $(V_{17} \oplus V_{20} \oplus V_{21})^{\otimes 2}$}
\begin{minipage}[t]{.45\textwidth}
\strut\vspace*{-\baselineskip}\newline
\begin{center}
\begin{tabular}{|c|c|c|}
\hline
Tensor & $V_i$ & Generators \\ \hline\hline
\multirow{4}{*}{$V_{17} \otimes V_{17}$}
  & $V_0$ & $x_1x_2 - x_2x_1$ \\ \cline{2-3}
  & $V_2$ & $x_1x_2 + x_2x_1$ \\ \cline{2-3}
  & $V_6$ & $x_1^2 + x_2^2$ \\ \cline{2-3}
  & $V_8$ & $x_1^2 - x_2^2$ \\ \hline\hline
\multirow{4}{*}{$V_{20} \otimes V_{20}$}
  & $V_0$ & $y_1y_2 + y_2y_1$ \\ \cline{2-3}
  & $V_3$ & $y_1y_2 - y_2y_1$ \\ \cline{2-3}
  & $V_5$ & $y_1^2 + y_2^2$ \\ \cline{2-3}
  & $V_8$ & $y_1^2 - y_2^2$ \\ \hline\hline
\multirow{4}{*}{$V_{21} \otimes V_{21}$}
  & $V_0$ & $z_1z_2 - z_2z_1$ \\ \cline{2-3}
  & $V_3$ & $z_1z_2 + z_2z_1$ \\ \cline{2-3}
  & $V_6$ & $z_1^2 - z_2^2$ \\ \cline{2-3}
  & $V_7$ & $z_1^2 + z_2^2$ \\ \hline\hline
\multirow{4}{*}{$V_{17} \otimes V_{20}$}
  & \multirow{2}{*}{$V_{11}$} & $x_2y_1 - \i x_1y_2$ \\ \cline{3-3} 
  &                           & $\i x_2y_2 - x_1y_1$ \\ \cline{2-3}
  & \multirow{2}{*}{$V_{13}$} & $x_2y_1 + \i x_1y_2$ \\ \cline{3-3} 
  &                           & $\i x_2y_2 +  x_1y_1$ \\ \hline\hline
\multirow{4}{*}{$V_{20} \otimes V_{17}$}
  & \multirow{2}{*}{$V_{11}$} & $y_1x_1 + \i y_2x_2$ \\ \cline{3-3} 
  &                           & $-\i y_2x_1 - y_1x_2$ \\ \cline{2-3}
  & \multirow{2}{*}{$V_{13}$} & $y_1x_1 - \i y_2x_2$ \\ \cline{3-3} 
  &                           & $-\i y_2x_1 +  y_1x_2$ \\ \hline
\end{tabular}
\end{center}
\end{minipage}
\begin{minipage}[t]{.45\textwidth}
\strut\vspace*{-\baselineskip}\newline
\begin{center}
\begin{tabular}{|c|c|c|}
\hline
Tensor & $V_i$ & Generators \\ \hline\hline
\multirow{4}{*}{$V_{17} \otimes V_{21}$}
  & \multirow{2}{*}{$V_{10}$} & $x_2z_1 + \i x_1z_2$ \\ \cline{3-3} 
  &                           & $\i x_2z_2 - x_1z_1$ \\ \cline{2-3}
  & \multirow{2}{*}{$V_{12}$} & $x_2z_1 - \i x_1z_2$ \\ \cline{3-3} 
  &                           & $\i x_2z_2 +  x_1z_1$ \\ \hline\hline
\multirow{4}{*}{$V_{21} \otimes V_{17}$}
  & \multirow{2}{*}{$V_{10}$} & $\i z_1x_1 + z_2x_2$ \\ \cline{3-3} 
  &                           & $z_2x_1 - \i z_1x_2$ \\ \cline{2-3}
  & \multirow{2}{*}{$V_{12}$} & $\i z_1x_1 - z_2x_2$ \\ \cline{3-3} 
  &                           & $z_2x_1 +  \i z_1x_2$ \\ \hline\hline
\multirow{4}{*}{$V_{20} \otimes V_{21}$}
  & \multirow{2}{*}{$V_{4}$} & $y_1z_2 + \i y_2z_1$ \\ \cline{3-3} 
  &                          & $\i y_2z_1 - y_1z_2$ \\ \cline{2-3}
  & \multirow{2}{*}{$V_{9}$} & $y_2z_2 +\i y_1z_1$ \\ \cline{3-3} 
  &                          & $\i y_1z_1 - y_2z_2$ \\ \hline\hline
\multirow{4}{*}{$V_{21} \otimes V_{20}$}
  & \multirow{2}{*}{$V_{4}$} & $z_2y_1 + \i z_1y_2$ \\ \cline{3-3} 
  &                          & $\i z_1y_2 - z_2y_1$ \\ \cline{2-3}
  & \multirow{2}{*}{$V_{9}$} & $z_2y_2 +\i z_1y_1$ \\ \cline{3-3} 
  &                          & $\i z_1y_1 - z_2y_2$ \\ \hline
\end{tabular}
\end{center}
\end{minipage}
\label{tensorSquareD4}
\end{table}
\vspace*{1em}

We will use Lemma \ref{DQ4decomps} that gives the additional tensor product decompositions needed in our proof that $V$ is an inner faithful
representation of $D(\mathfrak{D}_{4})$.  Its proof is a computation that can be obtained from the $S$-matrix of $D(\mathfrak{D}_{4})$ or by a direct computation that 
we leave to the reader.
\begin{lemma} \label{DQ4decomps}
One has the following direct sum decompositions:
\begin{eqnarray*}
V_4 \otimes V_{20} & \cong & V_{19} \oplus V_{21}, \\
V_{11} \otimes V_{17} & \cong & V_{18} \oplus V_{20}, \\
V_3 \otimes V_{17} & \cong & V_{15}, \\
V_4 \otimes V_{17} & \cong & V_{14} \oplus V_{16}, \\
V_2 \otimes V_3 & \cong & V_1. \\
\end{eqnarray*}
\end{lemma}

\begin{proposition}\label{inner-faithful-n-2}
    The representation $V=V_{17}\oplus V_{20}\oplus V_{21}$ is an inner faithful
    representation of $D(\mathfrak{D}_{4})$.
\end{proposition}
\begin{proof}
By Table \ref{tensorSquareD4}, the simple modules 
\begin{gather*} V_0,V_2,V_3,V_4,V_5,V_6,V_7,V_8,V_9,V_{10},V_{11},V_{12},V_{13}
\end{gather*}
all appear in the decomposition of $V\otimes V$.  By Lemma \ref{DQ4decomps}, the 
simple modules
\begin{gather*} V_{14},V_{15},V_{16}, V_{18},V_{19},
\end{gather*}
all appear in the decomposition of $V_{4}\otimes V_{17}$, $V_{4}\otimes V_{20}$, 
$V_{11}\otimes V_{17}$, and $V_{3}\otimes V_{17}$, which appear in the 
decomposition of $V\otimes V \otimes V$. The remaining simple module
\begin{gather*} V_{1} = V_{2}\otimes V_{3}
\end{gather*}
appears in the decomposition of $V\otimes V\otimes V \otimes V$.
\end{proof}

Next let $\constb,\constc,\constd,\conste \in \Bbbk$ be arbitrary and nonzero.
In Table \ref{tensorSquareD4},
by identifying generators that play the role of $u$ and $v$ in
Table \ref{irrepsD4}, we
have identified an isomorphism between the two copies of a simple
$D(\mathfrak{D}_{4})$-module that appear in the tensor products of the
simple $D(\mathfrak{D}_{4})$-modules in the two orders.  So, using Table
\ref{tensorSquareD4}, the span of the list
of elements in Table \ref{relationsQ4} is a $D(\mathfrak{D}_4)$-submodule of $V\otimes V$, where
the $D(\mathfrak{D}_{4})$-module corresponding to the span of the
elements in a row of the display is listed to the left.  Here $u_1,u_2,u_3,u_4$ are nonzero constants, 
and $\alpha,\beta,\gamma \in \{1,-1\}$.
Further, in some cases the representation corresponding to the relation depends upon the value of the parameter; the notation $V_{i,j}$ indicates when the 
parameter in the relation is equal to $1$ the simple module corresponds to $V_i$, and when the parameter is $-1$, the relation corresponds 
to $V_j$. So, for example, $V_{0,2}$ means that $x_1x_2 - x_2x_1$ corresponds to the representation $V_0$ and $x_1x_2 + x_2 x_1$ corresponds to the representation $V_2$.

\begin{table}[H]
\caption{Relations for an algebra with action of $D(\mathfrak{D}_{4})$}
\begin{minipage}[t]{.45\textwidth}
\strut\vspace*{-\baselineskip}\newline
\scalebox{.75}{
\begin{tabular}{|c|c|} 
\hline
Rep & Relations \\ \hline
$V_{0,2}$ & $x_1x_2 - \alpha x_2x_1$ \\ \hline
$V_{3,0}$ & $y_1y_2 - \beta y_2y_1$ \\ \hline
$V_{0,3}$ & $z_1z_2 - \gamma z_2z_1$ \\ \hline
\multirow{2}{*}{$V_{11}$} & $(x_2y_1 - \i x_1y_2) - \constb(y_1x_1 + \i y_2x_2)$ \\
                       & $(\i x_2y_2 - x_1y_1) - \constb(-\i y_2x_1 - y_1x_2)$ \\ \hline
\multirow{2}{*}{$V_{13}$} & $(x_2y_1 + \i x_1y_2) - \constb(y_1x_1 - \i y_2x_2)$ \\
                       & $(\i x_2y_2 + x_1y_1) - \constb(-\i y_2x_1 + y_1x_2)$ \\ \hline
\end{tabular}
}
\end{minipage}
\begin{minipage}[t]{.45\textwidth}
\strut\vspace*{-\baselineskip}\newline
\scalebox{.75}{
\begin{tabular}{|c|c|} 
\hline
Rep & Relations \\ \hline
\multirow{2}{*}{$V_{10}$} & $(x_2z_1 + \i x_1z_2) - \constc(\i z_1x_1 + z_2x_2)$ \\
                       & $(\i x_2z_2 - x_1z_1) - \constc(z_2x_1 - \i z_1x_2)$ \\ \hline
\multirow{2}{*}{$V_{12}$} & $(x_2z_1 - \i x_1z_2) - \constc(\i z_1x_1 - z_2x_2)$ \\
                       & $(\i x_2z_2 + x_1z_1) - \constc(z_2x_1 + \i z_1x_2)$ \\ \hline
\multirow{2}{*}{$V_{4}$}  & $(y_1z_2 + \i y_2z_1) - \constd(z_2y_1 + \i z_1y_2)$ \\
                       & $(\i y_2z_1 - y_1z_2) - \constd(\i z_1y_2 - z_2y_1)$ \\ \hline
\multirow{2}{*}{$V_{9}$}  & $(y_2z_2 + \i y_1z_1) - \conste(z_2y_2 + \i z_1y_1)$ \\
                       & $(\i y_1z_1 - y_2z_2) - \conste(\i z_1y_1 - z_2y_2)$ \\ \hline
\end{tabular}
}
\end{minipage}
\label{relationsQ4}
\end{table}

We may choose a different generating set for this span, which no
longer makes it evident that it is a $D(\mathfrak{D}_4)$-module, but makes the
quotient ring easier to describe.  Indeed, by taking the sum and difference
of the relations from the pairs of representations $(V_{11},V_{13})$,
$(V_{10},V_{12})$ or across the same representation $V_4$ and $V_9$, (and dividing by constants),
we obtain the relations that appear in Theorem \ref{double-ore}.

To see a particular example, adding the first generators of the copies of $V_{13}$ and
$V_{11}$ in $V \otimes V$ gives us $2x_2y_1 - 2\constb y_1x_1$ and subtracting gives us
$2\i x_1y_2 + 2\constb \i y_2x_2$ which is equivalent to $x_1y_2 + \constb y_2x_2$.

\begin{remark}
We could have chosen also to use the other relations
appearing as summands of the representations $V_{17}^{\otimes 2}$,
$V_{20}^{\otimes 2}$ and $V_{21}^{\otimes 2}$.  Some of these seem to give relations of an
AS regular algebra, but we did not explore these cases in detail.
\end{remark}

\begin{theorem}\label{double-ore}
Let $V = V_{17} \oplus V_{20} \oplus V_{21}$ be the representation of
$D(\mathfrak{D}_4)$ given by the table on the left below:

\begin{minipage}[t]{.45\textwidth}
\strut\vspace*{-\baselineskip}\newline 
\begin{tabular}{|c|c|c|c|c|} \hline
$V_i$ & Gen. & Gr. & $r$ & $s$ \\ \hline\hline
{\multirow{2}{*}{$V_{17}$}}
& $x_1$ & $s$    & $x_2$ & $-\i x_1$ \\ 
& $x_2$ & $sr^2$ & $-x_1$ & $\i x_2$ \\ \hline
{\multirow{2}{*}{$V_{20}$}}
& $y_1$ & $sr$   & $-y_2$ & $y_2$ \\ 
& $y_2$ & $sr^3$ & $-y_1$ & $y_1$ \\ \hline
{\multirow{2}{*}{$V_{21}$}}
& $z_1$ & $sr$   & $-\i z_2$ & $z_2$ \\ 
& $z_2$ & $sr^3$ & $-\i z_1$ & $-z_1$ \\ \hline
\end{tabular}
\end{minipage}
\begin{minipage}[t]{.45\textwidth}
\strut\vspace*{-\baselineskip}\newline
\begin{tabular}{|c|c|} 
\hline
\multicolumn{2}{|c|}{Relations} \\ \hline \hline
$x_1x_2 - \alpha x_2x_1$     & $y_1y_2 - \beta y_2y_1$ \\ \hline
$z_1z_2 - \gamma z_2z_1$     & $x_1y_1 - \constb y_1x_2$ \\ \hline
$x_2y_1 - \constb y_1x_1$    & $y_1z_2 - \constd z_2y_1$ \\ \hline
$x_1y_2 + \constb y_2x_2$    & $x_2y_2 + \constb y_2x_1$ \\ \hline
$y_2z_2 - \conste z_2y_2$    & $x_1z_1 - \i \constc z_1x_2$ \\ \hline
$x_2z_1 - \i \constc z_1x_1$ & $y_1z_1 - \conste z_1y_1$ \\ \hline
$\i x_1z_2 - \constc z_2x_2$ & $\i x_2z_2 - \constc z_2x_1$ \\ \hline
$y_2z_1 - \constd z_1y_2$ & \\ \hline
\end{tabular}
\end{minipage}

Let $B$ be the quotient of $\Bbbk\langle x_1, x_2, y_1,y_2, z_1, z_2 \rangle$
by the relations in Table \ref{relationsQ4}, or equivalently, in the table
on the right above.  Then for any nonzero $\constb,\constc,\constd,\conste$ and
$\alpha, \beta, \gamma \in \{\pm 1\}$, the algebra $B$ is an
Artin-Schelter regular algebra of dimension six upon which $D(\mathfrak{D}_4)$
acts inner faithfully.

More precisely,
$B \cong \Bbbk_\mathbf{q}[y_1,y_2,z_1,z_2]_\alpha[x_1,x_2;\sigma]$, where
the skewing matrix $\mathbf{q}$ and the double Ore extension data
$\sigma$ are given by:

\begin{minipage}[t]{.3\textwidth}
\strut\vspace*{-\baselineskip}\newline 
\begin{equation*}
\mathbf{q} = \begin{pmatrix}
1 & \beta & \conste^{-1} & \constd^{-1} \\
\beta & 1 & \constd^{-1} & \conste^{-1} \\
\conste & \constd & 1 & \gamma \\
\constd & \conste & \gamma & 1
\end{pmatrix},
\end{equation*}
\end{minipage}
\quad\quad
\begin{minipage}[t]{.5\textwidth}
\strut\vspace*{-\baselineskip}\newline 
\scalebox{.9}{
\begin{tabular}{cc}
$\sigma(y_1) = \constb \begin{pmatrix} 0 & y_1 \\ y_1 & 0\end{pmatrix}$, & 
$\sigma(y_2) = \constb \begin{pmatrix} 0 & -y_2 \\ -y_2 & 0\end{pmatrix}$, \\
& \\
$\sigma(z_1) = \constc \begin{pmatrix} 0 & \i z_1 \\ \i z_1 & 0\end{pmatrix}$, & 
$\sigma(z_2) = \constc \begin{pmatrix} 0 & -\i z_2 \\ -\i z_2 & 0\end{pmatrix}$.
\end{tabular}}
\end{minipage}
\end{theorem}

\begin{proof}
Let $A$ be the subalgebra generated by the images of $y_1,y_2,z_1,z_2$ in $B$.
Checking that the relations in the table above describe the relations of $A$,
as well as give rise to the double Ore extension $B_\alpha[x_1,x_2;\sigma]$,
may be verified directly.  

Note that each of these matrices is skew diagonal, so by Remark \ref{diagDOData},
one needs only check that the matrices corresponding to $\sigma_{12}$ and $\sigma_{21}$
commute.  As one can see, these matrices are diagonal and therefore satisfy the 
requirements of the remark.  In particular, these assignments define an algebra 
homomorphism $\sigma : A \to M_2(A)$, so that the data
$\{\sigma,\alpha\}$ give a well-defined double Ore extension.
\end{proof}

In addition to the description given in Theorem \ref{double-ore},
$B$ is a derivation quotient algebra defined
by a twisted superpotential $\mathbf{w}_B$ according to \cite{DV}.
We recall now how one may compute the twisted superpotential and the corresponding
Nakayama automorphism of $B$ using the Frobenius structure of $B^!$.
Indeed, suppose $B = T(V)/\langle I \rangle$ is a Koszul AS-regular algebra of dimension $n$ 
(presented with $I \subseteq V \otimes V$) so that $B^! = T(V^*)/\langle I^\perp \rangle$.  
Given a basis $x_1,\dots,x_n$ of $V$, let $x_1^*,\dots,x_n^*$ be the corresponding dual basis
of $V^*$.  Further, we identify both $x_j$ and $x_j^*$ with their images in $B$ and $B^!$, 
respectively.  

Since $B^!$ is Frobenius, there is a nondegenerate $n$-form on $B_1^!$ given by the multiplication map
$$\Phi : B_1^! \otimes \cdots \otimes B_1^! \to B_n^! \cong \Bbbk$$
where the last isomorphism is made explicit via choosing a basis $\{\vartheta\}$ of $B_n^!$.
Then $B$ is the derivation-quotient algebra corresponding to the twisted superpotential given by
$$\omega = \sum_{1 \leq j_1,\dots,j_n \leq n} \Phi(x_{j_1}^*,\dots,x_{j_n}^*) x_{j_1}\cdots x_{j_n}.$$
That is, the coefficient of $x_{j_1}\cdots x_{j_n}$ in the superpotential $\omega$ defining $B$
may be taken to be the scalar $c_{j_1,\dots,j_n}$ such that
$x_{j_1}^*\cdots x_{j_n}^* = c_{j_1,\dots,j_n}\vartheta$, where $\vartheta$ was the choice of
basis of $B_n^!$ made above.

Returning to our case, a Gr\"obner basis calculation shows that $x_1^*x_2^*y_1^*y_2^*z_1^*z_2^*$ is
nonzero in $B_6^!$, and a calculation shows that the superpotential of $B$ has the form:
$$\mathbf{w}_B =  x_1x_2y_1y_2z_1z_2 + \beta\gamma x_1x_2y_2y_1z_2z_1 - \alpha\beta\gamma x_2x_1y_2y_1z_2z_1+\cdots$$
where we suppress the remaining terms for brevity.  We note that all
nonzero terms appearing in the superpotential may be shown to have identity 
group grade.
One may also check that the Nakayama automorphism $\mu$ associated to the superpotential
$\mathbf{w}_B$ is given by
$$\mu = \begin{pmatrix}
-\alpha\constb^{-2}\constc^{-2} I & 0 & 0 \\
 0 & -\alpha\beta\constb^2\constd^{-1}\conste^{-1} I & 0 \\
 0 & 0 & \alpha\gamma\constc^2\constd\conste I\end{pmatrix}$$
where $I$ is the $2\times 2$ identity matrix.  

The reason for us to introduce the superpotential here is so that we may appeal to results of 
\cite[Theorem 3.2]{C1}, which allow us use the superpotential to compute the homological determinant
$\hdet_B: D(\mathfrak{D}_4)\rightarrow \Bbbk$ of the action of $D(\mathfrak{D}_4)$ on $B$.  Indeed, the 
action of $D(\mathfrak{D}_4)$ on $V$ induces an action on $V^{\otimes 6}$ in such a way that the span of
$\mathbf{w}_B$ is a one-dimensional representation $W$ of $D(\mathfrak{D}_4)$ whose character is precisely
the homological determinant of the action.  Furthermore, our remark on the
group grade of the superpotential also shows that the group grade portion of
the homological determinant of this action is trivial

Hence in order to complete our computation of the homological determinant of the 
action, it suffices to determine the action of $r$ and $s$ on $\mathbf{w}_B$.
Recall that the actions of $r$ and $s$ on $V$ are given by the matrices
\begin{equation}\label{r-s-action}
\begin{pmatrix}
0 & 1 & 0 & 0 & 0 & 0 \\
-1 & 0 & 0 & 0 & 0 & 0 \\
0 & 0 & 0 & -1 & 0 & 0 \\
0 & 0 & -1 & 0 & 0 & 0 \\
0 & 0 & 0 & 0 & 0 & -\i \\
0 & 0 & 0 & 0 & -\i & 0 \end{pmatrix}
\text{ and }
\begin{pmatrix}
-\i & 0 & 0 & 0 & 0 & 0 \\
0 & \i & 0 & 0 & 0 & 0 \\
0 & 0 & 0 & 1 & 0 & 0 \\
0 & 0 & 1 & 0 & 0 & 0 \\
0 & 0 & 0 & 0 & 0 & 1 \\
0 & 0 & 0 & 0 & -1 & 0 \end{pmatrix},
\end{equation}
respectively.  The actions of $r$ and $s$ on $\mathbf{w}_B$ are given by
\begin{eqnarray*}
\mathbf{w}_B\cdot r &=& (x_2)(-x_1)(-y_2)(-y_1)(-\i z_2)(-\i z_1) \\
                    & &  + \beta\gamma(x_2)(-x_1)(-y_1)(-y_2)(-\i z_2)(-\i z_1) \\
                    & &  - \alpha\beta\gamma(x_1)(-x_2)(-y_1)(-y_2)(-\i z_1)(-\i z_2)+\cdots \\
&=&x_2x_1y_2y_1z_2z_1 + \beta\gamma x_2x_1y_1y_2 z_2z_1 - \alpha\beta\gamma x_1x_2y_1y_2 z_1 z_2+\cdots \text{ and }\\
\mathbf{w}_B\cdot s &=& -\i x_1(\i x_2)(y_2)(-y_1)( z_2)(-z_1) \\
                    & & +\beta\gamma(-\i x_1)(\i x_2)(y_1)(y_2)( -z_1)( z_2)\\
                    & & - \alpha\beta\gamma(\i x_2)(-\i x_1)(y_1)(y_2)( -z_1)( z_2)+\cdots \\
&=&x_1x_2y_2y_1z_2z_1 - \beta\gamma x_1x_2y_1y_2 z_1z_2 + \alpha\beta\gamma x_2x_1y_1y_2 z_1 z_2+\cdots.
\end{eqnarray*}
Note that since the above matrices are monomial, the only contribution to the
$x_1x_2y_1y_2z_1z_2$ term appearing in $\mathbf{w}_B\cdot r$ (respectively $\mathbf{w}_B \cdot s$)
arises from the image of the third (respectively, second) term under $r$ (respectively, $s$).
We have therefore proven the following result.

\begin{proposition} \label{prop:hdetQ4}
Let $D(\mathfrak{D}_4)$ act on the algebra $B$ as in the statement of Theorem \ref{double-ore}.
Then the group grade component of the homological determinant of this action is trivial, and one has that
$$\hdet_B(r) = - \alpha\beta\gamma \qquad\text{and}\qquad  \hdet_B(s) = - \beta\gamma.$$
In particular, the action has a trivial homological determinant if and only if
$\alpha=1$ and $\beta = -\gamma$.
\end{proposition}
\end{section}

\begin{section}{Invariants for \texorpdfstring{$D(\mathfrak{D}_4)$}{D(D4)} acting on \texorpdfstring{$B$}{B} with trivial homological determinant}
\label{invariants}
We now begin our computation of the invariant ring corresponding to a family of 
algebras arising from this example.


\begin{proposition}\label{invariant-equation}
Let $D(\mathfrak{D}_4)$ act on the algebra $B$ as in the statement of Theorem \ref{double-ore}.
A basis of $B$ is given by $\{\mathbf{x}^{\mathbf{a}}= x_1^{a_1}x_2^{a_2}y_1^{b_1}y_2^{b_2}z_1^{c_1}z_2^{c_2}\}$ where $a_1,a_2,b_1,b_2,c_1,c_2$ are nonnegative integers. If a monomial $\mathbf{x}^{\mathbf{a}}\in B$ is a term in an element of the invariant subring $B^{D(\mathfrak{D}_4)}$, then $a_1\pm a_2, b_1\pm b_2, c_1\pm c_2$ are even and $(a_1-a_2)\equiv b_1-b_2+c_2-c_2\mod 4$.
\end{proposition}
\begin{proof}
From Lemma \ref{identity-component}, we are interested first in the identity component $B_e$ and then
the restricted action of $\mathfrak{D}_4$ on $B_e$.  That a basis of $B$ is given by monomials of the form
$x_1^{a_1}x_2^{a_2}y_1^{b_1}y_2^{b_2}z_1^{c_1}z_2^{c_2}$ is clear from the definition of the skew
polynomial ring and double Ore extensions.  Such an element has group grade
\begin{eqnarray*}
s^{a_1}(s^{-1})^{a_2}(sr)^{b_1}((sr)^{-1})^{b_2}(sr)^{c_1}((sr)^{-1})^{c_2} & = & 
s^{a_1 - a_2}(sr)^{b_1 - b_2 + c_1 - c_2}.
\end{eqnarray*}
In order for this element to have grade equal to $e$, we must have that
$a_1 - a_2 \equiv b_1 + c_1 - b_2 - c_2 \mod 4$, and both sides must be even.  It follows that
$a_1 \pm a_2$ is even as well since $\langle s \rangle \cap \langle sr \rangle = \langle s^2 \rangle$.
Note also that, if a monomial of $B$ is part of an invariant under the subgroup $\langle s^2 \rangle$ of $G$, we must also have that $c_1 \pm c_2$ is even.  It follows that a monomial involved in an invariant of $D(\mathfrak{D}_4)$
must have $b_1 \pm b_2$ is even.
\end{proof}

\begin{definition}
Let $\mathbf{a}=(a_1,a_2,b_1,b_2,c_1,c_2)\in \mathbb{N}^6$ and let $\mathbf{x}^{\mathbf{a}} =  x_1^{a_1}x_2^{a_2}y_1^{b_1}y_2^{b_2}z_1^{c_1}z_2^{c_2}$ be a monomial in $B$. There is an action of elements of $\mathfrak{D}_4$ on exponent vectors $\mathbf{a}=(a_1,a_2,b_1,b_2,c_1,c_2)$ given by
\begin{align*}
\mathbf{a}\cdot r &=(a_1,a_2,b_1,b_2,c_1,c_2)\cdot r = (a_2,a_1,b_2,b_1,c_2,c_1),\\
\mathbf{a}\cdot s &=(a_1,a_2,b_1,b_2,c_1,c_2)\cdot s = (a_1,a_2,b_2,b_1,c_2,c_1).
\end{align*}
This leads to a general action of the elements $g\in\mathfrak{D}_4$ on monomials given by 
$$\mathbf{x}^{\mathbf{a}}\cdot g = \lambda(\mathbf{a}, g)\mathbf{x}^{\mathbf{a}\cdot g}$$
where $\lambda(\mathbf{a}, g)$ are scalars. Furthermore, $\mathbf{x}^{\mathbf{a}}\cdot gh = (\lambda(\mathbf{a}, g)\mathbf{x}^{\mathbf{a}\cdot g})\cdot h = \lambda(\mathbf{a}, g)\lambda(\mathbf{a}\cdot g,h)\mathbf{x}^{\mathbf{a}\cdot gh}$, and $\mathbf{x}^{\mathbf{a}}\cdot gh = \lambda(\mathbf{a}, gh)\mathbf{x}^{\mathbf{a}\cdot gh}$ so the scalars satisfy
$$\lambda(\mathbf{a}, gh) = \lambda(\mathbf{a}, g)\lambda(\mathbf{a}\cdot g, h).$$
\end{definition} 
If we assume that the monomial $\mathbf{x}^{\mathbf{a}}$ is a term of an invariant, we can determine the scalars $\lambda(\mathbf{a}, g)$ for when $g$ is equal to $r$, $s$ or $rs$ using Proposition \ref{invariant-equation} in the following way. 
\begin{align*}
\mathbf{x}^{\mathbf{a}}\cdot r &= (x_1^{a_1}x_2^{a_2}y_1^{b_1}y_2^{b_2}z_1^{c_1}z_2^{c_2})\cdot r = (-1)^{a_1+b_1+b_2}(-\i)^{c_1+c_2}x_2^{a_1}x_1^{a_2}y_2^{b_1}y_1^{b_2}z_2^{c_1}z_1^{c_2} \\
&= (-1)^{a_1+b_1+b_2+c_1+c_2}(\i)^{c_1+c_2}\alpha^{a_1a_2}\beta^{b_1b_2}\gamma^{c_1c_2}x_1^{a_2}x_2^{a_1}y_1^{b_2}y_2^{b_1}z_1^{c_2}z_2^{c_1} \\
&= ((-1)^{a_1}(\i)^{c_1+c_2}\alpha^{a_1}\beta^{b_1}\gamma^{c_1})x_1^{a_2}x_2^{a_1}y_1^{b_2}y_2^{b_1}z_1^{c_2}z_2^{c_1} =\lambda(\mathbf{a}, r) \mathbf{x}^{\mathbf{a}\cdot r},\\
&\text{ hence, } \lambda(\mathbf{a}, r) = (-1)^{a_1}\i^{c_1+c_2}\alpha^{a_1}\beta^{b_1}\gamma^{c_1}.
\end{align*}
\begin{align*}
\mathbf{x}^{\mathbf{a}}\cdot s &= (x_1^{a_1}x_2^{a_2}y_1^{b_1}y_2^{b_2}z_1^{c_1}z_2^{c_2})\cdot s = (-1)^{c_1}(-\i)^{a_1}(\i)^{a_2}x_1^{a_1}x_2^{a_2}y_2^{b_1}y_1^{b_2}z_2^{c_1}z_1^{c_2} \\
&= (-1)^{a_1+c_1}(\i)^{a_1+a_2}\beta^{b_1}\gamma^{c_1}x_1^{a_1}x_2^{a_2}y_1^{b_2}y_2^{b_1}z_1^{c_2}z_2^{c_1} =\lambda(\mathbf{a}, s) \mathbf{x}^{\mathbf{a}\cdot s},\\
&\text{ hence, } \lambda(\mathbf{a}, s) = (-1)^{a_1+c_1}\i^{a_1+a_2}\beta^{b_1}\gamma^{c_1}, \text{ and }
\end{align*}
\begin{align*}
\mathbf{x}^{\mathbf{a}}\cdot rs  &= ((-1)^{a_1}(\i)^{c_1+c_2}\alpha^{a_1}\beta^{b_1}\gamma^{c_1})(x_2^{a_1}x_1^{a_2}y_2^{b_1}y_1^{b_2}z_2^{c_1}z_1^{c_2})\cdot s \\
&= (-1)^{a_1+a_2+c_2}(\i)^{c_1+c_2+a_1+a_2} \alpha^{a_1}\beta^{b_1+b_2}\gamma^{c_1+c_2}    x_1^{a_2}x_2^{a_1}y_1^{b_1}y_2^{b_2}z_1^{c_1}z_2^{c_2} \\
&= (-1)^{c_1}(\i)^{c_1+c_2+a_1+a_2} \alpha^{a_1} x_1^{a_2}x_2^{a_1}y_1^{b_1}y_2^{b_2}z_1^{c_1}z_2^{c_2} =\lambda(\mathbf{a}, rs) \mathbf{x}^{\mathbf{a}\cdot rs},\\
&\text{ hence, } \lambda(\mathbf{a}, rs) = (-1)^{c_1}\i^{c_1+c_2+a_1+a_2} \alpha^{a_1}.
\end{align*}

By Proposition \ref{invariant-equation}, we may order elements in the basis
of $B$ in graded lexicographic order with $x_1 > x_2 > y_1 > y_2 > z_1 > z_2$.
Based on the action of $\mathfrak{D}_4$ on
the exponent vectors given above, it is natural to restrict our attention
to those exponent vectors which are lead monomials in their orbit under this action.  
This is achieved for those exponent vectors in the following set:

\begin{equation} \label{defX}
X = \left\{(a_1,a_2,b_1,b_2,c_1,c_2) \in \mathbb{N}^6~\left|~~ \begin{matrix}a_1 \geq a_2,
b_1 \geq b_2, \\ a_1\pm a_2,b_1\pm b_2, c_1\pm c_2~\text{are even, and} \\ a_1-a_2\equiv b_1-b_2+c_1-c_2\mod 4\end{matrix}\right.\right\}.
\end{equation}
In particular, we may assume that $a_1 \geq a_2$ since $rs$ interchanges only these
two values, and we can assume that $b_1 \geq b_2$ since $s$ interchanges both
$b_1$ and $b_2$ as well as $c_1$ and $c_2$ while keeping $a_1$ and $a_2$ fixed.
That is, we cannot assume anything about the relative order of $c_1$ and
$c_2$ since we cannot independently flip these components of the exponent vector.

\begin{definition} \label{def:fa}
For a fixed exponent vector
$\mathbf{a}=(a_1,a_2,b_1,b_2,c_1,c_2) \in X$ defined as
in \eqref{defX}, we let $f_\mathbf{a}$ denote the
following invariant:
$$f_{\mathbf{a}} =  \mathbf{x}^{\mathbf{a}} + \lambda(\mathbf{a}, r) \mathbf{x}^{\mathbf{a}\cdot r}+\lambda(\mathbf{a}, s) \mathbf{x}^{\mathbf{a}\cdot s}+\lambda(\mathbf{a}, rs) \mathbf{x}^{\mathbf{a}\cdot rs}.$$
\end{definition}
To see that $f_\mathbf{a}$ is an invariant, note that
\begin{eqnarray*}
f_{\mathbf{a}}\cdot r &=&  \lambda(\mathbf{a}, r)\mathbf{x}^{\mathbf{a}\cdot r} + \lambda(\mathbf{a}, r)\lambda(\mathbf{a}\cdot r, r) \mathbf{x}^{\mathbf{a}\cdot r^2} + \\
                      & &  \lambda(\mathbf{a}, s)\lambda(\mathbf{a}\cdot s, r) \mathbf{x}^{\mathbf{a}\cdot sr} + \lambda(\mathbf{a}, rs)\lambda(\mathbf{a}\cdot rs,r)\mathbf{x}^{\mathbf{a}\cdot rsr} \\
                      &=&  \lambda(\mathbf{a}, r)\mathbf{x}^{\mathbf{a}\cdot r} + \lambda(\mathbf{a}, r)\lambda(\mathbf{a}\cdot r, r) \mathbf{x}^{\mathbf{a}} + \\
                      & &  \lambda(\mathbf{a}, s)\lambda(\mathbf{a}\cdot s, r) \mathbf{x}^{\mathbf{a}\cdot rs}+\lambda(\mathbf{a}, rs)\lambda(\mathbf{a}\cdot rs,r) \mathbf{x}^{\mathbf{a}\cdot s}\\
                      &=&  \lambda(\mathbf{a}, r^2)\mathbf{x}^{\mathbf{a}} + \lambda(\mathbf{a}, r) \mathbf{x}^{\mathbf{a}\cdot r} + 
                           \lambda(\mathbf{a}, rsr) \mathbf{x}^{\mathbf{a}\cdot s}+\lambda(\mathbf{a}, sr)\mathbf{x}^{\mathbf{a}\cdot rs}\\
                      &=&  \mathbf{x}^{\mathbf{a}} + \lambda(\mathbf{a}, r) \mathbf{x}^{\mathbf{a}\cdot r} +
                           \lambda(\mathbf{a}, s) \mathbf{x}^{\mathbf{a}\cdot s}+\lambda(\mathbf{a}, rs) \mathbf{x}^{\mathbf{a}\cdot rs} = f_{\mathbf{a}}.
\end{eqnarray*}
The last equality was obtained from the fact that $\lambda(\mathbf{a}, r^2)=(-1)^{a_1+a_2+c_1+c_2}=1$ and since the action of $rsr$ and $s$ are the same on $\mathbf{x}^{\mathbf{a}}$ so $\lambda(\mathbf{a}, rsr)=\lambda(\mathbf{a}, s)$. Similarly the action of $rs$ and $sr$ are the same on $\mathbf{x}^{\mathbf{a}}$ so $\lambda(\mathbf{a}, rs)=\lambda(\mathbf{a}, sr)$. 

Since the set of exponent vectors $\{\mathbf{a},\mathbf{a}\cdot r, \mathbf{a}\cdot s,\mathbf{a}\cdot rs\}$ is closed under the action of $\mathfrak{D}_4$, every invariant $f\in B$ can be written 
uniquely as
\begin{equation}\label{invariant}
f = \sum_{\mathbf{a}\in X} c_{\mathbf{a}}f_{\mathbf{a}} 
\end{equation}
where $c_{\mathbf{a}}$ are scalars and the set $X$ is defined
in \eqref{defX}.

\begin{lemma}
Let $\mathbf{a},\mathbf{a}'\in X$ such that $\mathbf{x}^{\mathbf{a}},\mathbf{x}^{\mathbf{a}'}\in B$ satisfy the conditions of Proposition \ref{invariant-equation}. Then $\mathbf{a}+\mathbf{a}'\in X$ and $\mathbf{x}^{\mathbf{a}}\mathbf{x}^{\mathbf{a}'} = \lambda \mathbf{x}^{\mathbf{a}+\mathbf{a}'}$ for some scalar $\lambda$.
\end{lemma}
\begin{proof}
Let $\mathbf{a},\mathbf{a}'\in X$ and let $\mathbf{x}^{\mathbf{a}}= x_1^{a_1}x_2^{a_2}y_1^{b_1}y_2^{b_2}z_1^{c_1}z_2^{c_2}$ and $\mathbf{x}^{\mathbf{a}'}= x_1^{a_1'}x_2^{a_2'}y_1^{b_1'}y_2^{b_2'}z_1^{c_1'}z_2^{c_2'}$ be two monomials in $B$. Using the relations in $B$, we would have
\begin{align*}
\mathbf{x}^{\mathbf{a}}\cdot \mathbf{x}^{\mathbf{a}'} &= (x_1^{a_1}x_2^{a_2}y_1^{b_1}y_2^{b_2}z_1^{c_1}z_2^{c_2})\cdot(x_1^{a_1'}x_2^{a_2'}y_1^{b_1'}y_2^{b_2'}z_1^{c_1'}z_2^{c_2'})\\
& = \lambda ~ x_1^{a_1+ a_1'}x_2^{a_2+a_2'}y_1^{b_1+b_1'}y_2^{b_2+b_2'}z_1^{c_1+c_1'}z_2^{c_2+c_2'} \\
&= \lambda ~ \mathbf{x}^{\mathbf{a}+\mathbf{a}'},\quad \text{where}\qquad \lambda\in \Bbbk.
\end{align*}
As a consequence of  Proposition \ref{invariant-equation}, one has that
$(z_1^{c_1}z_2^{c_2})x_1^{a_1'} = \mu_1 x_1^{a_1'}(z_1^{c_1}z_2^{c_2})$, with $\mu_1\in \Bbbk$ and
$(y_1^{b_1}y_2^{b_2})x_1^{a_1'} = \mu_2 x_1^{a_1'}(y_1^{b_1}y_2^{b_2})$ with $\mu_2\in \Bbbk$ because $c_1+c_2$, $b_1+b_2$
are both even.  A similar result holds if we replace
$x_1^{a_1'}$ by $x_2^{a_2'}$ in the above expression and we get new scalars $\mu_1',\mu_2'$.
In the end we obtain the desired result with some scalars $\mu_1\mu_2\mu_1'\mu_2'\alpha^i\beta^j\gamma^t=\lambda$
for some integers $i,j,t$.  If $a_1\pm a_2$ and $a_1'\pm a_2'$ are even, so also is $(a_1 + a_1')\pm (a_2 +  a_2').$
A similar result holds for $b_1\pm b_2$ and $c_1\pm c_2$.
If $a_1-a_2\equiv b_1-b_2+c_1-c_2 \mod 4$ and $a_1'-a_2'\equiv b_1'-b_2'+c_1'-c_2' \mod 4$, it follows that
$(a_1+a_1')-(a_2+a_2')\equiv (b_1+b_1')-(b_2+b_2')+(c_1+c_1')-(c_2+c_2') \mod 4$, so  $\mathbf{a}+\mathbf{a}'\in X$.
\end{proof}

The following result concerns a noncommutative analogue of SAGBI bases, which will be useful for
proving that we have found a generating set of $B^{D(\mathfrak{D}_4)}$.

\begin{lemma}\label{sagbi-bases}
Let $A$ be an $\mathbb{N}$-graded algebra with PBW basis $\{\mathbf{x}^{\mathbf{a}}\}$. Fix a graded term order $<$ on
the exponent vectors $\mathbf{a}$. Let $\mathcal{B}\subseteq A$ be a graded subalgebra of $A$ and $f_1,f_2,\ldots,f_r$
be homogeneous elements of $\mathcal{B}$ whose leading monomial has exponent vectors $\mathbf{a}_1,\mathbf{a}_2,\ldots,\mathbf{a}_r$
and such that the leading monomial of $f_1^{m_1}f_2^{m_2}\cdots f_r^{m_r}$ is equal to
$\mathbf{x}^{m_1\mathbf{a}_1+m_2\mathbf{a}_2+\cdots+m_r\mathbf{a}_r}$ for some
nonnegative integers $m_1,m_2,\cdots, m_r$. Suppose that for every $g\in\mathcal{B}$, the exponent vector
of the leading monomial of $g$ is in the submonoid generated by
$\mathbf{a}_1,\mathbf{a}_2,\ldots,\mathbf{a}_r$.  Then $f_1,f_2,\cdots,f_r$ generate $\mathcal{B}$.
\end{lemma}
\begin{proof}
Since $\mathcal{B}$ is a graded subalgebra and the $f_i$ are homogeneous, it suffices to assume that $g\in\mathcal{B}$ is homogeneous
of degree $d$.  We prove that $g$ is in the subalgebra generated by $f_1,\dots,f_r$ using induction on the
(well-ordered) term order on exponent vectors.  

By hypothesis, there exist nonnegative integers $n_1,\dots,n_r$ such that the leading monomial of
$g$ is equal to the leading monomial of $f_1^{n_1}\cdots f_r^{n_r}$.  Therefore
there exists a scalar $\lambda$ such that $g - \lambda f_1^{n_1}\cdots f_r^{n_r}$ has smaller
leading monomial than $g$.  We may repeat this process, which must terminate since the
term order is graded.  It follows that $g$ is in the subalgebra generated by the $f_i$.
\end{proof}

We remark that the term order appearing in the previous lemma need not be the same as the term
order one uses to compute a Gr\"obner basis of the defining ideal of $A$.  Indeed,
this result is most useful when one has an algebra $A$ with a PBW basis $\{\mathbf{x}^\mathbf{a}\}$
such that the concept of ``lead monomial" with respect to this basis of $A$ is not well-behaved (i.e.
the lead monomial of $fg$ is not necessarily the product of the lead monomials of $f$ and $g$),
but the lead term \emph{is} well-defined for elements of the \emph{subalgebra} $\mathcal{B}$.



\subsection{The special case of \texorpdfstring{$B$}{B} with \texorpdfstring{$(\alpha,\beta,\gamma, \constb,\constc,\constd,\conste) = (1,1,-1,1,1,1,1)$}{(alpha,beta,gamma,u1,u2,u3,u4) = (1,1,-1,1,1,1,1)}}

First, note that in this case the action has a trivial homological determinant
if we specify $\alpha=1,\beta=1$ and $\gamma=-1$.
In this case, the values of the scalars $\lambda(\mathbf{a}, -)$ obtained from the action 
of elements of $\mathfrak{D}_4$ on a monomial
$\mathbf{x}^{\mathbf{a}} =  x_1^{a_1}x_2^{a_2}y_1^{b_1}y_2^{b_2}z_1^{c_1}z_2^{c_2}$ 
satisfying the conditions of Proposition \ref{invariant-equation} are given by
\begin{gather}\label{invariant-coefficients}
\begin{aligned}
&\lambda(\mathbf{a}, r)  = (-1)^{a_1+c_1c_2}\i^{c_1+c_2} = (-1)^{a_1+c_1}\i^{c_1+c_2}, \\
&\lambda(\mathbf{a}, s)  = (-1)^{a_1+c_1+c_1c_2}\i^{a_1+a_2} = (-1)^{a_1}\i^{a_1+a_2},\\
&\lambda(\mathbf{a}, rs) = (-1)^{c_1}\i^{c_1+c_2+a_1+a_2} = (-1)^{c_1}\i^{c_1+c_2+a_1+a_2}.
\end{aligned}
\end{gather}
Therefore by equations \eqref{invariant} and
\eqref{invariant-coefficients}, we need only consider
invariants of the form
\begin{equation}\label{orbit-sum}
\begin{aligned}
 f_{\mathbf{a}} &=  \mathbf{x}^{\mathbf{a}} + \lambda(\mathbf{a}, r) \mathbf{x}^{\mathbf{a}\cdot r}+\lambda(\mathbf{a},s) \mathbf{x}^{\mathbf{a}\cdot s}+\lambda(\mathbf{a},rs) \mathbf{x}^{\mathbf{a}\cdot rs}\\
&= x_1^{a_1}x_2^{a_2}y_1^{b_1}y_2^{b_2}z_1^{c_1}z_2^{c_2} +(-1)^{a_1+c_1}\i^{c_1+c_2} x_1^{a_2}x_2^{a_1}y_1^{b_2}y_2^{b_1}z_1^{c_2}z_2^{c_1}\\
&+(-1)^{a_1}\i^{a_1+a_2}x_1^{a_1}x_2^{a_2}y_1^{b_2}y_2^{b_1}z_1^{c_2}z_2^{c_1}\\
&+ (-1)^{c_1}\i^{c_1+c_2+a_1+a_2}x_1^{a_2}x_2^{a_1}y_1^{b_1}y_2^{b_2}z_1^{c_1}z_2^{c_2}.
\end{aligned}
\end{equation}

Next, we record the following result which shows that there are no new generators for the invariant subring
$B^{D(\mathfrak{D}_4)}$ beyond degree 14 in this case.
The result is a consequence of a generalized version of an upper bound due to Broer that appears in \cite[Lemma 2.2]{KKZ1}.

\begin{lemma} \label{lem:Broer}
The degree of the generators of the invariant subring $B^{D(\mathfrak{D}_4)}$ is bounded by 14 if $(\alpha,\beta,\gamma, \constb,\constc,\constd,\conste) = (1,1,-1,\constb,\constc,\constd,\conste)$ and $\constd^{16}=\conste^{16}$.
\end{lemma}
\begin{proof}
The Drinfeld double ${D(\mathfrak{D}_4)}$ is a semisimple Hopf algebra acting on $B$ whose 
dimension is 6. It can be verified that the subalgebra
$C=\Bbbk[y_2y_1, z_2z_1,x_2^2x_1^2, x_1^4+x_2^4, y_1^4+y_2^4,z_1^4+z_2^4]$ of the invariant 
subring $B^{D(\mathfrak{D}_4)}$ is an iterated Ore extension if $\constd^{16}=\conste^{16}$. 
Indeed, the elements $f_1=y_2y_1, f_2= z_2z_1, f_3=x_2^2x_1^2,$ 
$f_4=x_1^4+x_2^4, f_5=y_1^4+y_2^4, f_6=z_1^4+z_2^4$ satisfy
$f_if_j=\alpha_{ij}f_jf_i$ for some scalars $\alpha_{ij}$ for all $1 \leq i < j \leq 6$
except $(i,j) = (5,6)$, which holds if and only if one has $\constd^{16}=\conste^{16}$.
Therefore, by \cite[Lemma 2.2]{KKZ1}, the degrees of a minimal set of generators of 
$B^{D(\mathfrak{D}_4)}$ is bounded by $\sum_{i=1}^n \deg(f_i)- \GKdim B = 20-6 = 14$.
\end{proof}

We now compute the invariant subring for the action 
of $D(\mathfrak{D}_4)$ on the algebra $B$ in the 
case $(\alpha,\beta,\gamma,u_1,u_2,u_3,u_4) = 
(1,1,-1,1,1,1,1)$, so that $B$ is the quotient of $\Bbbk\langle x_1,x_2,y_1,y_2,z_1,z_2 \rangle$ by the relations
\begin{align*}
&x_1x_2 -  x_2x_1~,~y_1y_2 -  y_2y_1~,~z_1z_2 + z_2z_1~,~  x_2y_1 - y_1x_1, ~ x_1y_2 + y_2x_2, \\
& x_2y_2 +  y_2x_1~, ~ x_1y_1 - y_1x_2~, ~ x_2z_1 - \i z_1x_1~, ~\i x_1z_2 - z_2x_2 ~, ~\i x_2z_2 - z_2x_1,\\
& x_1z_1 - \i z_1x_2~,~  y_2z_1 - z_1y_2~, ~ y_1z_2 - z_2y_1~,~ y_1z_1 - z_1y_1~,~ y_2z_2 - z_2y_2.
\end{align*}
The remainder of this section is devoted to proving the following theorem:

\begin{theorem} \label{generatorTheorem}
Let $D(\mathfrak{D}_4)$ act on the algebra $B$ as in the statement of Theorem \ref{double-ore},
with $(\alpha,\beta,\gamma,u_1,u_2,u_3,u_4) = (1,1,-1,1,1,1,1)$.  Then the following elements minimally generate
$B^{D(\mathfrak{D}_4)}$ as a $\Bbbk$-algebra:

\begin{minipage}[t]{.48\textwidth}
\strut\vspace*{-\baselineskip}\newline 
\begin{center}
\begin{tabular}{| c | c |}
\hline
 & Generators \\ \hline \hline
$g_1$ & $y_1y_2$ \\ \hline
$g_2$ & $z_1z_2$ \\ \hline
$g_3$ & $x_1^2x_2^2$ \\ \hline
$g_4$ & $x_1^4 + x_2^4$ \\ \hline
$g_5$ & $y_1^4 + y_2^4$ \\ \hline
$g_6$ & $z_1^4 + z_2^4$ \\ \hline
$g_7$ & $(x_1^2 - x_2^2)(y_1^2 - y_2^2)$ \\ \hline
$g_8$ & $(x_1^2 + x_2^2)(z_1^2 - z_2^2)$ \\ \hline
\end{tabular}
\end{center}
\end{minipage}
\begin{minipage}[t]{.48\textwidth}
\strut\vspace*{-\baselineskip}\newline 
\begin{center}
\begin{tabular}{| c | c |}
\hline
 & Generators \\ \hline \hline
$g_9$    & $(x_1x_2)(x_1^4 - x_2^4)$ \\ \hline
$g_{10}$ & $(x_1x_2)(x_1^2 + x_2^2)(y_1^2 - y_2^2)$ \\ \hline
$g_{11}$ & $(x_1x_2)(x_1^2 - x_2^2)(z_1^2 - z_2^2)$ \\ \hline
$g_{12}$ & $(x_1x_2)(y_1^2z_1^2 + y_2^2z_2^2)$ \\ \hline
$g_{13}$ & $(x_1x_2)(y_1^2z_2^2 + y_2^2z_1^2)$ \\ \hline
$g_{14}$ & $(x_1^4 - x_2^4)(y_1^2z_2^2 + y_2^2z_1^2)$ \\ \hline
$g_{15}$ & $(x_1^2 + x_2^2)(y_1^4z_2^2 - y_2^4z_1^2)$ \\ \hline
$g_{16}$ & $(x_1^2 - x_2^2)(y_1^2z_2^4 - y_2^2z_1^4)$ \\ \hline
$g_{17}$ & $y_1^4z_2^4 + y_2^4z_1^4$ \\ \hline
\end{tabular}
\end{center}
\end{minipage}
\end{theorem}

We begin with a brief discussion of the proof strategy.  First, a direct calculation
using the action of $r$ and $s$ given by the matrices \eqref{r-s-action} shows that the
elements $g_i$ are all invariant under the $\mathfrak{D}_4$ action and are also of
identity group grade.  By Lemma \ref{sagbi-bases}, it is enough to show that any invariant
has the same leading monomial as some other element in the subalgebra generated by
the $g_i$.  By the discussion after Definition \ref{def:fa}, it
suffices to prove that this holds for invariants of the form $f_\mathbf{a}$
for $\mathbf{a} \in X$.  Also, by Lemma \ref{lem:Broer} it suffices to show
prove that this holds for those elements $f_\mathbf{a}$ of degree at most 14.
Next, we shrink the set $X$ we must consider by using the generators $g_1,g_2$ and $g_3$.
We finally conclude with a case-by-case analysis of the exponent vectors remaining.

We now give the following lemma, which will aid us in the
proof of the proposition that follows.
\begin{lemma}\label{invariant-structure}
Let $B$ be a domain, and let $G$ be a group.  Let $f,g,h \in B$
with $f = gh$ and $g \neq 0$.
\begin{enumerate}
\item If $G$ acts on $B$ and $f,g \in B^G$, then $h \in B^G$.
\item If $B$ is $G$-graded and $f,g \in B_e$ then $h \in B_e$.
\end{enumerate}
In particular, if $D(G)$ acts on $B$ and $f,g \in B^{D(G)}$, then
$h \in B^{D(G)}$.
\end{lemma}
\begin{proof}
First the first claim, note that if $f = gh$ and $g$ are both in $B^G$, then
for any $v \in G$, one has that
$$0 = f - f = gh - (gh)\cdot v = (g\cdot v)(h\cdot v) - gh = g(h\cdot v - h)$$
so that the result follows if $B$ is a domain.

For the second claim, write $f = f_e$, $g = g_e$ and $h = \sum_{w \in G} h_w$.
Then $gh = g_e\sum_{w \in G} h_w = \sum_{w \in G} g_eh_w = f_e$.  This implies
that for every $w \neq e$, one has $g_eh_w = 0$ and hence $h_w = 0$ for $w \neq e$
as desired.  The final remark follows from Lemma \ref{identity-component}.
\end{proof}

\begin{proposition}\label{main-cases}
Let $\aa \in X$.  Then $f_\aa = \pm m f_{\aa'}$ for some  
monomial invariant $m$ in the subalgebra generated by
$g_1,g_2$ and $g_3$, and
$\aa' \in (X_1 \cup X_2 \cup X_3 \cup X_4) \subseteq X$
where
\begin{align*}
X_{1} &= \{(2a, 0 ,2b,0,2c,0)\;|\; a,b,c\in\mathbb{N}\}, \\
X_{2} &= \{(2a, 0 ,2b,0,0, 2c)\;|\; a,b,c\in\mathbb{N}, b,c >0 \},\\
X_{3} &= \{(2a-1, 1 ,2b,0,2c,0)\;|\; a,b,c\in\mathbb{N}, a > 0\}, \text{ and }\\
X_{4} &= \{(2a-1, 1 ,2b,0,0,2c)\;|\; a,b,c\in\mathbb{N}, a,b,c > 0\}.
\end{align*}
\end{proposition}

\begin{proof}
Let $\aa \in X$.  Recall that $f_\aa$ as defined below is an invariant:
$$f_{\mathbf{a}} =  \mathbf{x}^{\mathbf{a}} + \lambda(\mathbf{a}, r) \mathbf{x}^{\mathbf{a}\cdot r}+\lambda(\mathbf{a}, s) \mathbf{x}^{\mathbf{a}\cdot s}+\lambda(\mathbf{a}, rs) \mathbf{x}^{\mathbf{a}\cdot rs}.$$
Since $a_1\geq a_2$, the monomial $\mathbf{x}^{\mathbf{a}}$ can be written as
$$\xx^\aa = x_1^{a_1}x_2^{a_2}y_1^{b_1}y_2^{b_2}z_1^{c_1}z_2^{c_2} =
\begin{cases}
  (x_1^2x_2^2)^{\frac{a_2}{2}}(x_1^{a_1-a_2}y_1^{b_1}y_2^{b_2}z_1^{c_1}z_2^{c_2}) & \text{if}~a_2~\text{is even}, \\
  (x_1^2x_2^2)^{\frac{a_2-1}{2}}(x_1^{a_1-a_2+1}x_2y_1^{b_1}y_2^{b_2}z_1^{c_1}z_2^{c_2}) & \text{if}~a_2~\text{is odd.}
  \\
\end{cases}$$
Since $a_1 - a_2$ is even, the exponent vector of the right factor in the above
expressions is either $(2a,0,b_1,b_2,c_1,c_2)$ for some $a \in \mathbb{N}$, or $(2a-1,1,b_1,b_2,c_1,c_2)$
for some $a \in \mathbb{N}$ with $a > 0$.  Note that since $x_1^2x_2^2$ is an invariant but $x_1x_2$ is not,
we must consider the case that there is a nonzero value in the $x_2$ coordinate of the exponent.

Next, since $b_1 \geq b_2$ and since $y_1y_2$ is an invariant,
we may factor out all $y_2$ terms by removing the corresponding $y_1$ term to the left of the monomial.  We therefore
have the factorizations:
$$\xx^\aa = 
\begin{cases}
  (x_1^2x_2^2)^{\frac{a_2}{2}}(y_1y_2)^{b_2}(x_1^{a_1-a_2}y_1^{b_1-b_2}z_1^{c_1}z_2^{c_2}) & \text{if}~a_2~\text{is even}, \\
  (-1)^{b_2}(x_1^2x_2^2)^{\frac{a_2-1}{2}}(y_1y_2)^{b_2}(x_1^{a_1-a_2+1}x_2y_1^{b_1-b_2}z_1^{c_1}z_2^{c_2}) & \text{if}~a_2~\text{is odd.}
  \\
\end{cases}$$
Finally, we may use a power of $z_1z_2$ to reduce to the case
when there is only one nonzero component in the last two components 
of the exponents we must consider.  Indeed, we also have the 
following factorizations, where we have replaced
$y_1y_2$, $z_1z_2$, and $x_1^2x_2^2$ by $g_1,g_2$ and $g_3$, 
respectively, in order to conserve space:
$$\xx^\aa = 
\begin{cases}
  (-1)^{\binom{c_2}{2}}g_3^{\frac{a_2}{2}}g_1^{b_2}g_2^{c_2}(x_1^{a_1-a_2}y_1^{b_1-b_2}z_1^{c_1-c_2}) & \text{if}~a_2~\text{is even,}~c_1 \geq c_2, \\
  (-1)^{\binom{c_1}{2}}g_3^{\frac{a_2}{2}}g_1^{b_2}g_2^{c_1}(x_1^{a_1-a_2}y_1^{b_1-b_2}z_2^{c_2-c_1}) & \text{if}~a_2~\text{is even,}~c_1 > c_2, \\
  (-1)^{b_2+\binom{c_2}{2}}g_3^{\frac{a_2-1}{2}}g_1^{b_2}g_2^{c_2}(x_1^{a_1-a_2+1}x_2y_1^{b_1-b_2}z_1^{c_1-c_2}) & \text{if}~a_2~\text{is odd,}~c_1 \geq c_2, \\
  (-1)^{b_2+\binom{c_1}{2}}g_3^{\frac{a_2-1}{2}}g_1^{b_2}g_2^{c_1}(x_1^{a_1-a_2+1}x_2y_1^{b_1-b_2}z_2^{c_2-c_1}) & \text{if}~a_2~\text{is odd,}~c_1 > c_2.
  \\
\end{cases}$$
Now let $m$ be the monomial invariant which is the product of 
the powers of $g_1,g_2,g_3$ appearing as the left factor in 
the previous display, and let
$\aa'$ be the exponent vector of the remaining factor on the right.  Note that
$\aa' \in X_1\cup X_2\cup X_3\cup X_4$ as in the statement of the theorem.
One checks that $\mathbf{x}^{\aa.g} = \pm m\mathbf{x}^{\aa'.g}$ for all
$g$ in $\mathfrak{D}_4$, and that the sign
that appears is independent of $g$.  Hence one has
$$f_\aa = \pm m(x^{\aa'} + \lambda(\aa,r)x^{\aa'\cdot r} + \lambda(\aa,s)x^{\aa'\cdot s} + \lambda(\aa,rs)x^{\aa'\cdot rs}).$$
Remark \ref{invariant-structure} shows that the
parenthesized expression on the right is an invariant,
and, in particular, it is equal to $f_{\aa'}$.
\end{proof}

It now suffices to show that given any $\aa \in X_j$ for
$j = 1,\dots,4$ such that the degree of $\xx^\aa$ is at most
14, there exists an invariant generated by those 
elements in Theorem \ref{generatorTheorem} whose lead term is
$\xx^\aa$.  We proceed with a case-by-case analysis.

\textit{\textbf{Case 1}: Exponent vector is in $X_{1}$, so has form $(2a,0,2b,0,2c,0)$ for some $a,b,c \in \mathbb{N}$:}
In this case, the invariant is given by
\begin{equation}
f_{\mathbf{a}}=  x_1^{2a}y_1^{2b}z_1^{2c} +(-1)^{c} x_2^{2a}y_2^{2b}z_2^{2c}+(-1)^{a}x_1^{2a}y_2^{2b}z_2^{2c}+ (-1)^{a+c}x_2^{2a}y_1^{2b}z_1^{2c}.
\end{equation}
Recall that $a$ and $b + c$ have the same parity based on
the definition of $X_1 \subseteq X$.  The following table gives 
the generators that we will use in this case:\\
~\\
\begin{center}\begin{tabular}{| c | c | c|} \hline
 & Generators of invariant ring & Exponent vector of lead monomial \\ \hline \hline
              $g_4$       &  $x_1^4+x_2^4$ & $(4,0,0,0,0,0)$ \\ \hline
                $g_5$     &  $y_1^4+y_2^4$  & $(0,0,4,0,0,0)$ \\ \hline
                $g_6$     &  $z_1^4+z_2^4$ & $(0,0,0,0,4,0)$\\ \hline
                $g_7$     &  $(x_1^2-x_2^2)(y_1^2-y_2^2)$  & $(2,0,2,0,0,0)$\\ \hline
                 $g_8$    & $(x_1^2 + x_2^2)(z_1^2 - z_2^2)$ & $(2,0,0,0,2,0)$\\ \hline
\end{tabular}\end{center}

If $a,b$ and $c$ are even, use $g_4,g_5$ and $g_6$ 
to create an element whose lead term is $\xx^\aa$. In this case, 
for example, the lead monomial of $f_{\mathbf{a}}$ has exponent 
vector $(4a',0,4b',0,4c',0)$ which can be written as a nonnegative 
integer combination of the exponent vectors of lead monomials of $g_4,g_5,g_6$, namely 
$a'(4,0,0,0,0,0)+b'(0,0,4,0,0,0)+c'(0,0,0,0,4,0)$. 

If $a$ is even, and $b$ and $c$ are odd, then $a = 0$ is not possible, otherwise, $f_{\mathbf{a}}=0$. Write $a=2a', b=2b'+1, c=2c'+1$, so the lead monomial of $f_{\mathbf{a}}$ has exponent vector $(4a',0,4b'+2,0,4c'+2,0)$, which can be written as a nonnegative integer combination of the exponent vectors of the lead monomials of $g_4$, $g_7g_8$, $g_5$ and $g_6$, namely
$(a'-1)(4,0,0,0,0,0)+(4,0,2,0,2,0)+b'(0,0,4,0,0,0)+c'(0,0,0,0,4,0)$.

If $a$ is odd, $b$ even and $c$ odd, write $a=2a'+1, b=2b', c=2c'+1$, so the lead monomial of $f_{\mathbf{a}}$ has exponent vector $(4a'+2,0,4b',0,4c'+2,0)$ which can be written as a nonnegative integer combination of the lead monomials of $g_4,g_5,g_6,g_8$.  That is: $a'(4,0,0,0,0,0)+b'(0,0,4,0,0,0) +c'(0,0,0,0,4,0)+(2,0,0,0,2,0)$.

If $a$ is odd, $b$ odd and $c$ even, write $a=2a'+1, b=2b'+1, c=2c'$, so the lead monomial of $f_{\mathbf{a}}$ has exponent vector $(4a'+2,0,4b'+2,0,4c',0)$ which can be written as a nonnegative integer combination of the lead monomials of $g_4,g_5,g_6,g_7$.  That is: $a'(4,0,0,0,0,0)+b'(0,0,4,0,0,0)+c'(0,0,0,0,4,0)+(2,0,2,0,0,0)$. 

\textit{\textbf{Case 2}: Exponent vector is in $X_{2}$, so has form $(2a,0,2b,0,0,2c)$ with $a,b,c \in \mathbb{N}$ and $b,c > 0$:} In this case, the invariant is given by
\begin{equation*}
f_{\mathbf{a}} =  x_1^{2a}y_1^{2b}z_2^{2c} +(-1)^{c} x_2^{2a}y_2^{2b}z_1^{2c}+(-1)^{a}x_1^{2a}y_2^{2b}z_1^{2c}+ (-1)^{a+c}x_2^{2a}y_1^{2b}z_2^{2c}.
\end{equation*}
Once again, we have that $a$ and $b+c$ have the same parity based on the definition of $X_2 \subseteq X$.  The following table gives the possible generators we will use to produce an invariant whose lead term is $\xx^\aa$.\\
\begin{center}
\begin{tabular}{| c | c | c|} \hline
 & Generators used & Exponent vector \\ \hline \hline
 $g_4$    & $x_1^4+x_2^4$ & $(4,0,0,0,0,0)$ \\ \hline
 $g_5$    & $y_1^4+y_2^4$  & $(0,0,4,0,0,0)$ \\ \hline
 $g_6$    & $(x_1^2-x_2^2)(y_1^2-y_2^2)$  & $(2,0,2,0,0,0)$\\ \hline
 $g_{14}$ & $(x_1^4-x_2^4)(y_1^2z_2^2+y_2^2z_1^2)$ & $(4,0,2,0,0,2)$\\ \hline
 $g_{15}$ & $(x_1^2+x_2^2)(y_1^4z_2^2-y_2^4z_1^2)$ & $(2,0,4,0,0,2)$\\ \hline
 $g_{16}$ & $(x_1^2-x_2^2)(y_1^2z_2^4-y_2^2z_1^4)$ & $(2,0,2,0,0,4)$\\ \hline
 $g_{17}$ & $y_1^4z_2^4+y_2^4z_1^4$ & $(0,0,4,0,0,4)$\\ \hline
\end{tabular}
\end{center}
We will also require the following 
invariants in order to obtain some other
exponent vectors.  While they are not
generators, they are in the subalgebra
generated by them, as one can check with 
the equations which follow the table.
\begin{center}
\begin{tabular}{| c | c | c|} \hline
& Other invariants used & Exponent vector \\ \hline
                $g_{18}$     &  $y_1^4z_2^8+y_2^4z_1^8$  & $(0,0,4,0,0,8)$\\ \hline
                 $g_{19}$    & $(x_1^4 - x_2^4)(y_1^2z_2^6 + y_2^2 z_1^6)$ & $(4,0,2,0,0,6)$\\ \hline
                 $g_{20}$    & $(x_1^2+x_2^2)(y_1^4z_2^6 - y_2^4z_1^6)$ & $(2,0,4,0,0,6)$\\ \hline
\end{tabular}
\end{center}
\begin{eqnarray*}
g_{18} & = & g_{17}g_6 - g_5g_2^4, \\
g_{19} & = & g_{14}g_6 + g_7g_8g_2^2 + g_{14}g_2^2, \\
g_{20} & = & g_{15}g_6 + g_8g_5g_2^2 + g_{15}g_2^2.
\end{eqnarray*}

If $a$ is even, and $b$ and $c$ are even, we will consider the cases in which 
$b<c, b=c, b>c$.  If $b=c$, use $g_4$ and $g_{17}$ to create an invariant
whose lead term is $\xx^\aa$.   If $b>c$, use $g_{17}$ repeatedly to match the 
exponent for $z_2$, then use $g_4$ and $g_5$ for the remainder of the exponents 
present to create an invariant whose lead term is $\xx^\aa$.
If $b<c$, use $g_4$ to obtain the exponent for $x_1$.  To handle the remaining 
exponent, since $b>0$, the only exponent vector of this type of degree less or 
equal to 14 is $(0,0,4,0,0,8)$ which is the lead term of $g_{18}$. 

If $a$ is even, and $b$ and $c$ are odd; then $a=0$ is not possible, otherwise, 
$f_{\mathbf{a}}=0$.  The exponent vector is therefore of the form
$(4a',0,4b'+2,0,0,4c'+2)$.  If $b',c' > 0$, then we may use $g_{14}$ to
reduce the exponent to one handled by the previous bullet paragraph.
If $c' = 0$, then we may use $g_4,g_5$ and $g_{14}$ to create
an invariant of the desired exponent.  This leaves us with
the case $b' = 0$ and $c' \neq 0$, which would imply that $a' + c' \leq 2$.
Since $a' > 0$, our degree restriction leaves us only the case $a' = c' = 1$ which 
gives exponent vector $(4,0,2,0,0,6)$ which is covered by $g_{19}$.

If $a$ is odd, $b$ even and $c$ odd, we know that since $b>0$, the lead monomial 
has exponent vector $(2a,0,2b,0,0,2c)=(4a'+2,0,4b',0,0,4c'+2)$.  There are only 
four cases of such exponents with total degree less than or equal to $14$.  These 
are $(2,0,4,0,0,2)$, $(2,0,4,0,0,6)$, $(6,0,4,0,0,2)$ and $(2,0,8,0,0,2)$.  Such 
exponents may be obtained by using $g_4,g_5,g_{15}$ and $g_{20}$. 

If $a$ is odd, $b$ odd and $c$ even, write $a=2a'+1, b=2b'+1, c=2c'$, so that the 
lead monomial of $f_{\mathbf{a}}$ has exponent vector $(4a'+2,0,4b'+2,0,0,4c')$. 
Various cases arise here such that the total degree is less than or equal
to 14, which are obtained by using $g_4,g_5,g_6$ or $g_{16}$.


\textit{\textbf{Case 3}: Exponent vector is in $X_{3}$, so has form 
$(2a-1,1,2b,0,2c,0)$}: In this case 
$a$ and $b+c$ have opposite parity, $a,b,c \in \mathbb{N}$
with $a > 0$, and $f_\mathbf{a}$ is:
\begin{eqnarray*}
f_{\mathbf{a}} & = & x_1^{2a-1}x_2y_1^{2b}z_1^{2c} - (-1)^{c} x_1x_2^{2a-1}y_2^{2b}z_2^{2c} - \\
               &   & (-1)^{a}x_1^{2a-1}x_2y_2^{2b}z_2^{2c} + (-1)^{a+c}x_1x_2^{2a-1}y_1^{2b}z_1^{2c}.
\end{eqnarray*}
The following table gives the possible generators we will use
in this case:
\begin{center}
\begin{tabular}{| c | c | c|} \hline
 & Generators used & Exponent vector \\ \hline \hline
$g_9$    & $(x_1x_2)(x_1^4 - x_2^4)$ & $(5,1,0,0,0,0)$ \\ \hline
$g_{10}$ & $(x_1x_2)(x_1^2 + x_2^2)(y_1^2 - y_2^2)$ & $(3,1,2,0,0,0)$ \\ \hline
$g_{11}$ & $(x_1x_2)(x_1^2 - x_2^2)(z_1^2 - z_2^2)$ & $(3,1,0,0,2,0)$ \\ \hline
$g_{12}$ & $(x_1x_2)(y_1^2z_1^2 + y_2^2z_2^2)$ & $(1,1,2,0,2,0)$ \\ \hline
\end{tabular}
\end{center}

If $b=0$ and $c=0$, then $a$ must be odd, say $a = 2a'-1$. Having $a=1$ forces $f_{\mathbf{a}}=0$ so we assume that $a'\geq 2$. The exponent vector is therefore $(4a'-3,1,0,0,0,0) = (5,1,0,0,0,0) + (4a'-8,0,0,0,0,0)$.  The former exponent is $g_9$, and the latter exponent is a multiple of the exponent of $g_4$.

If $b$ and $c$ are both positive, then use $g_{12}$ and we are 
left with a term having lead monomial in the form 
$(2a-2,0,2b-2,0,2c-2,0)$ which has been considered in Case 1.

If $b>0$, $c=0$ and $a=1$, then the orbit sum
$f_{\mathbf{a}}=0$. So we can assume $a>1$. 
Using $g_{10}$ first, we are left with a leading monomial of 
the form $(2(a-2),0,2(b-1),0,0,0)$.
Since $a-2$ and $b-1$ are the same parity, this case has been 
treated in the first or fourth paragraph
of Case 1 taking $c=0$.

If $b=0$ and $c>0$ and $a=1$, the orbit sum $f_{\mathbf{a}}=0$, so we can assume $a>1$. 
Using $g_{11}$, we are left with a leading monomial of the form $(2(a-2),0,0,0,2(c-1),0)$.
Since $a-2$ and $c-1$ are the same parity, this case has been treated in the first or third paragraph of Case 1 taking $b=0$.

\textit{\textbf{Case 4}: Exponent vector is in $X_{4}$ so has 
form $(2a-1,1,2b,0,0,2c)$:}
In this case, $a$ and $b+c$ have opposite parity, $a,b,c$ are 
positive integers, and $f_\aa$ is
\begin{eqnarray*}
f_{\mathbf{a}} & = & x_1^{2a-1}x_2y_1^{2b}z_2^{2c} - (-1)^{c} x_1x_2^{2a-1}y_2^{2b}z_1^{2c}- \\
               &   & (-1)^{a}x_1^{2a-1}x_2y_2^{2b}z_1^{2c}- (-1)^{a+c}x_1x_2^{2a-1}y_1^{2b}z_2^{2c}.
\end{eqnarray*}

Since $c$ is positive, we may use $g_{13}$ and we are left with a term 
having lead monomial in the form $(2a-2,0,2b-2,0,0,2c-2)$ which has been 
considered in Case 2 if $b \geq 2$.  If $b = 1$, then the exponent vector 
is of the form $(2a-1,1,2,0,0,2c-2)$.  There are four exponent vectors of 
this form of degree less than or equal to 14 that may not be obtained 
using the exponent vectors of the lead terms of our given generators,
and they are given by the following invariants:
\begin{center}
\begin{tabular}{| c | c | c|} \hline
 & Other Invariants & Exponent vector of lead monomial \\
 \hline \hline
 $g_{21}$ & $(x_1x_2)(x_1^2+x_2^2)(y_1^2z_2^4 - y_2^2z_1^4)$ & $(3,1,2,0,0,4)$ \\ \hline
 $g_{22}$ & $(x_1x_2)(y_1^2z_2^6 + y_2^2z_1^6)$ & $(1,1,2,0,0,6)$ \\ \hline
 $g_{23}$ & $(x_1x_2)(x_1^2+x_2^2)(y_1^2z_2^8 - y_2^2z_1^8)$ & $(3,1,2,0,0,8)$ \\ \hline
 $g_{24}$ & $(x_1x_2)(y_1^2z_2^{10} + y_2^2z_1^{10})$ & $(1,1,2,0,0,10)$ \\ \hline
\end{tabular}
\end{center}
Each of these elements are in the subalgebra of the invariant ring 
generated by the elements in Theorem \ref{generatorTheorem}.  Indeed,
one has that:
\begin{eqnarray*}
g_{21} & = & \frac{1}{2}(g_{10}g_6 - g_8g_{12} - g_8g_{13}), \\
g_{22} & = & g_{13}g_6 + g_{12}g_2^2, \\
g_{23} & = & g_{21}g_6 - g_{10}g_2^4, \\
g_{24} & = & g_{22}g_6 - g_{13}g_2^4.
\end{eqnarray*}
This completes our proof of Theorem \ref{generatorTheorem}.\\

Since we chose parameters so that the homological determinant of the action of $D(\mathfrak{D}_4)$  on $B$ is trivial, the following corollary follows from \cite{KKZ0}.
\begin{corollary}
The ring of invariants $B^{D(\mathfrak{D}_4)}$ is an Artin-Schelter Gorenstein algebra.
\end{corollary}
\end{section}

\begin{section} {\texorpdfstring{$D(\mathfrak{D}_8)$}{D(D8)} the case \texorpdfstring{$n=4$}{n=4}: the dicyclic group of order 16}
\label{dicyclic-n-4}

We next consider the dicyclic group of order 16. Using the results of
\cite{W1, W2} we see that
there are 46 simple modules $V_i$, for $i=0,\ldots, 45$. Table \ref{simpleD8} 
shows the conjugacy classes (with choice of representative $a$) with 
centralizers, simple modules and transversals for each conjugacy class.

\begin{table}[H]
\caption{Simple $D(\mathfrak{D}_8)$-modules}
\begin{center} 
\begin{tabular}{|c|c|c|c|}
\hline
  Conj class & $C_G(a)$ & Rep $C_G(a)$& transveral\\
\hline
\hline
$\{a=e\}$ & $\mathfrak{D}_{8}$ & $\psi_0, \psi_1, \psi_2, \psi_3, \chi_1,\chi_2,\chi_3$& $\{e\}$\\
\hline
 &  & $V_0, V_1, V_2, V_3, V_4, V_5, V_6$ & \\ \hline \hline
$\{a=r^4\}$ & $\mathfrak{D}_{8}$ & $\psi_0, \psi_1, \psi_2, \psi_3, \chi_1,\chi_2,\chi_3 $& $\{e\}$\\
\hline
 &  & $V_7, V_8, V_9, V_{10}, V_{11}, V_{12}, V_{13}$ & \\
\hline \hline
$\{a=r, r^7\}$ & $\mathbb{Z}_8 = \langle r \rangle $ & $\alpha_0, \alpha_1, \alpha_2, \alpha_3, \alpha_4, \alpha_5, \alpha_6, \alpha_7$&$\{e, s\}$ \\
\hline
 &  & $V_{14}, V_{15}, V_{16}, V_{17}, V_{18}, V_{19}, V_{20}, V_{21}$ & \\
\hline \hline
$\{a=r^2, r^6\}$ & $\mathbb{Z}_8= \langle r \rangle $ & $\beta_0, \beta_1, \beta_2, \beta_3, \beta_4, \beta_5, \beta_6, \beta_7$ & $\{e, s\}$\\
\hline
&  & $V_{22}, V_{23}, V_{24}, V_{25}, V_{25}, V_{27}, V_{28}, V_{29}$ & \\
\hline \hline
$\{a=r^3, r^5\}$ & $\mathbb{Z}_8 = \langle r \rangle $ & $\gamma_0, \gamma_1, \gamma_2, \gamma_3, \gamma_4, \gamma_5, \gamma_6, \gamma_7$&$\{e, s\}$ \\
\hline
&  & $V_{30}, V_{31}, V_{32}, V_{33}, V_{34}, V_{35}, V_{36}, V_{37}$ & \\
\hline \hline
$\{a=s, sr^2, sr^4, sr^6\}$ & $\mathbb{Z}_4= \langle s \rangle $ & $\tau_0, \tau_1, \tau_2, \tau_3$ & $\{e, r, r^2, r^3\}$\\
\hline
&  & $V_{38}, V_{39}, V_{40}, V_{41}$ & \\
\hline \hline
$\{a=sr, sr^3, sr^5, sr^7\}$ & $\mathbb{Z}_4 = \langle sr \rangle $ & $\sigma_0, \sigma_1, \sigma_2, \sigma_3$&$\{e, r,r^2,r^3\}$ \\
\hline
&  & $V_{42}, V_{43}, V_{44}, V_{45}$ & \\
\hline
\end{tabular}
\end{center}
\label{simpleD8}
\end{table}

In Table \ref{simpleD8}, we use the following notation.  We let $\psi_0$ be the 
trivial representation of $\mathfrak{D}_8$, $\psi_1$ the representation
of $\mathfrak{D}_8$ given by $t\cdot r = t$ and $t\cdot s = -t$, $\psi_2$ the
representation of $\mathfrak{D}_8$ given by $t\cdot r = -t$ and
$t \cdot s = t$, and $\psi_3$  the representation of $\mathfrak{D}_8$
given by $t\cdot r = -t = t\cdot s$.
Letting $\omega =e^{\frac{\pi \i}{4}}$, we have three two-dimensional irreducible
representations for $\mathfrak{D}_{8}$ which given by

\begin{center}
\begin{tabular}{|c|c|c|c|}
\hline
Representation & Generator & $r$ & $s$ \\ \hline
\multirow{2}{*}{$\chi_1$} & $u$ & $\omega^2 u$ & $v$ \\
                          & $v$ & $\omega^6 u$ & $u$ \\ \hline
\multirow{2}{*}{$\chi_2$} & $u$ & $\omega^3 u$ & $-v$ \\
                          & $v$ & $\omega^5 u$ & $u$ \\ \hline
\multirow{2}{*}{$\chi_3$} & $u$ & $\omega u$ & $-v$ \\
                          & $v$ & $\omega^7 u$ & $u$ \\ \hline
\end{tabular}
\end{center}

In addition, $\alpha_i$ has the action $t\cdot(r^l)=\omega^{il}t$,
$\beta_i$ has the action $t\cdot(r^l)=\omega^{il}t$ and
$\gamma_i$ has the action $t\cdot(r^l)=\omega^{il}$ for $0\leq i\leq 7$, $0\leq l\leq 7$.  Finally, $\tau_{j}$ has the action $t\cdot(s^{m})=\i^{jm}t$ and $\sigma_{j}$ has the action $t\cdot((sr)^{m})=\i^{jm}t$ for $0\leq j\leq 3$, $0\leq m\leq 3$.

Table \ref{irrepsD8} describes the action of $D(\mathfrak{D}_8)$ 
on its simple modules.  By Notation \ref{not:modulesOverDouble}, 
these  are given by representations of the centralizers of elements of
$\mathfrak{D}_8$ induced to representations of $\mathfrak{D}_8$.  In 
Table \ref{irrepsD8}, we use the order that appears in Table 
\ref{simpleD8}. 

\begin{table}
\caption{Explicit action of $D(\mathfrak{D}_8)$ on its simple modules}
\begin{minipage}[t]{.45\textwidth}
\strut\vspace*{-\baselineskip}\newline 
\begin{center}
\scalebox{.80}{
\begin{tabular}{|c|c|c|c|c|c|}
\hline
Class & $V_i$ & Gen. & Gr. & $r$ & $s$ \\ \hline\hline
{\multirow{8}{*}{$[e]$}} &
     $V_0$ & $u$ & $e$ & $u$ & $u$ \\ \cline{2-6}
 &  $V_1$ & $u$ & $e$ & $u$ & $-u$ \\ \cline{2-6}
 &  $V_2$ & $u$ & $e$ & $-u$ & $u$ \\ \cline{2-6}
 &  $V_3$ & $u$ & $e$ & $-u$ & $-u$ \\ \cline{2-6}
 &  {\multirow{2}{*}{$V_4$}} 
      & $u$ & $e$ & $\omega^2u$ & $v$ \\
     && $v$ & $e$ & $\omega^6v$ & $u$ \\ \cline{2-6}
 &  {\multirow{2}{*}{$V_5$}} 
      & $u$ & $e$ & $\omega^3 u$ & $-v$ \\
     && $v$ & $e$ & $\omega^5v$ & $u$ \\  \cline{2-6}
 &  {\multirow{2}{*}{$V_6$}} 
      & $u$ & $e$ & $\omega u$ & $-v$ \\
     && $v$ & $e$ & $\omega^7v$ & $u$ \\ \hline\hline
{\multirow{8}{*}{$[r^4]$}} &
     $V_7$ & $u$ & $r^4$ & $u$ & $u$ \\ \cline{2-6}
 &  $V_8$ & $u$ & $r^4$ & $u$ & $-u$ \\ \cline{2-6}
 &  $V_9$ & $u$ & $r^4$ & $-u$ & $u$ \\ \cline{2-6}
 &  $V_{10}$ & $u$ & $r^4$ & $-u$ & $-u$ \\ \cline{2-6}
 &  {\multirow{2}{*}{$V_{11}$}} 
      & $u$ & $r^4$ & $\omega^2u$ & $v$ \\
     && $v$ & $r^4$ & $\omega^6v$ & $u$ \\ \cline{2-6}
 &  {\multirow{2}{*}{$V_{12}$}} 
      & $u$ & $r^4$ & $\omega^3u$ & $-v$ \\
     && $v$ & $r^4$ & $\omega^5v$ & $u$ \\  \cline{2-6}
 &  {\multirow{2}{*}{$V_{13}$}} 
      & $u$ & $r^4$ & $\omega u$ & $-v$ \\
     && $v$ & $r^4$ & $\omega^7v$ & $u$ \\ \hline\hline
{\multirow{16}{*}{$[r]$}} &
     {\multirow{2}{*}{$V_{14}$}}
      & $u$ & $r$   & $u$ & $v$ \\ 
     && $v$ & $r^7$ & $v$ & $u$ \\ \cline{2-6}
 &  {\multirow{2}{*}{$V_{15}$}}
      & $u$ & $r$   & $\omega u$  & $v$ \\ 
     && $v$ & $r^7$ & $\omega^7v$ & $-u$ \\ \cline{2-6}
 &  {\multirow{2}{*}{$V_{16}$}}
      & $u$ & $r$   & $\omega^2u$ & $v$ \\ 
     && $v$ & $r^7$ & $\omega^6v$ & $u$ \\ \cline{2-6}
 &  {\multirow{2}{*}{$V_{17}$}}
      & $u$ & $r$   & $\omega^3u$ & $v$ \\ 
     && $v$ & $r^7$ & $\omega^5v$ & $-u$ \\ \cline{2-6}
&  {\multirow{2}{*}{$V_{18}$}}
      & $u$ & $r$   & $-u$ & $v$ \\ 
     && $v$ & $r^7$ & $-v$ & $u$ \\ \cline{2-6}
 &  {\multirow{2}{*}{$V_{19}$}}
      & $u$ & $r$   & $\omega^5u$ & $v$ \\ 
     && $v$ & $r^7$ & $\omega^3v$ & $-u$ \\ \cline{2-6}
&  {\multirow{2}{*}{$V_{20}$}}
      & $u$ & $r$   & $\omega^6u$ & $v$ \\ 
     && $v$ & $r^7$ & $\omega^2v$ & $u$ \\ \cline{2-6}
 &  {\multirow{2}{*}{$V_{21}$}}
      & $u$ & $r$   & $\omega^7u$ & $v$ \\ 
     && $v$ & $r^7$ & $\omega v$ & $-u$ \\ \hline\hline
{\multirow{16}{*}{$[r^2]$}}&
     {\multirow{2}{*}{$V_{22}$}}
      & $u$ & $r^2$   & $u$ & $v$ \\ 
     && $v$ & $r^6$ & $v$ & $u$ \\ \cline{2-6}
 &  {\multirow{2}{*}{$V_{23}$}}
      & $u$ & $r^2$   & $\omega u$  & $v$ \\ 
     && $v$ & $r^6$ & $\omega^7v$ & $-u$ \\ \cline{2-6}
 &  {\multirow{2}{*}{$V_{24}$}}
      & $u$ & $r^2$   & $\omega^2u$ & $v$ \\ 
     && $v$ & $r^6$ & $\omega^6v$ & $u$ \\ \cline{2-6}
 &  {\multirow{2}{*}{$V_{25}$}}
      & $u$ & $r^2$   & $\omega^3u$ & $v$ \\ 
     && $v$ & $r^6$ & $\omega^5v$ & $-u$ \\ \cline{2-6}
&  {\multirow{2}{*}{$V_{26}$}}
      & $u$ & $r^2$   & $-u$ & $v$ \\ 
     && $v$ & $r^6$ & $-v$ & $u$ \\ \cline{2-6}
 &  {\multirow{2}{*}{$V_{27}$}}
      & $u$ & $r^2$   & $\omega^5u$ & $v$ \\ 
     && $v$ & $r^6$ & $\omega^3v$ & $-u$ \\ \cline{2-6}
&  {\multirow{2}{*}{$V_{28}$}}
      & $u$ & $r^2$   & $\omega^6u$ & $v$ \\ 
     && $v$ & $r^6$ & $\omega^2v$ & $u$ \\ \cline{2-6}
 &  {\multirow{2}{*}{$V_{29}$}}
      & $u$ & $r^2$   & $\omega^7u$ & $v$ \\ 
     && $v$ & $r^6$ & $\omega v$ & $-u$ \\ \hline
\end{tabular}}
\end{center}
\end{minipage}
\begin{minipage}[t]{.45\textwidth}
\strut\vspace*{-\baselineskip}\newline 
\begin{center}
\scalebox{.80}{
\begin{tabular}{|c|c|c|c|c|c|}
\hline
Class & $V_i$ & Gen. & Gr. & $r$ & $s$ \\ \hline\hline
{\multirow{16}{*}{$[r^3]$}} &
     {\multirow{2}{*}{$V_{30}$}}
      & $u$ & $r^3$   & $u$ & $v$ \\ 
     && $v$ & $r^5$ & $v$ & $u$ \\ \cline{2-6}
 &  {\multirow{2}{*}{$V_{31}$}}
      & $u$ & $r^3$   & $\omega u$  & $v$ \\ 
     && $v$ & $r^5$ & $\omega^7 v$ & $-u$ \\ \cline{2-6}
 &  {\multirow{2}{*}{$V_{32}$}}
      & $u$ & $r^3$   & $\omega^2u$ & $v$ \\ 
     && $v$ & $r^5$ & $\omega^6v$ & $u$ \\ \cline{2-6}
 &  {\multirow{2}{*}{$V_{33}$}}
      & $u$ & $r^3$   & $\omega^3u$ & $v$ \\ 
     && $v$ & $r^5$ & $\omega^5v$ & $-u$ \\ \cline{2-6}
&  {\multirow{2}{*}{$V_{34}$}}
      & $u$ & $r^3$   & $-u$ & $v$ \\ 
     && $v$ & $r^5$ & $-v$ & $u$ \\ \cline{2-6}
 &  {\multirow{2}{*}{$V_{35}$}}
      & $u$ & $r^3$   & $\omega^5u$ & $v$ \\ 
     && $v$ & $r^5$ & $\omega^3v$ & $-u$ \\ \cline{2-6}
&  {\multirow{2}{*}{$V_{36}$}}
      & $u$ & $r^3$   & $\omega^6u$ & $v$ \\ 
     && $v$ & $r^5$ & $\omega^2v$ & $u$ \\ \cline{2-6}
 &  {\multirow{2}{*}{$V_{37}$}}
      & $u$ & $r^3$   & $\omega^7u$ & $v$ \\ 
     && $v$ & $r^5$ & $\omega v$ & $-u$ \\ \hline\hline
{\multirow{16}{*}{$[s]$}} &
     {\multirow{4}{*}{$V_{38}$}}
     & $u$ & $s$    & $v$  & $u$  \\
	&& $v$ & $sr^2$   & $p$ & $q$ \\ 
	&& $p$ & $sr^4$    & $q$  & $p$  \\  
    && $q$ & $sr^6$ &  $u$ & $v$ \\ \cline{2-6}
 &  {\multirow{4}{*}{$V_{39}$}}
     & $u$ & $sr$    &$v$  & $\i u$ \\ 
    && $v$ & $sr^2$    & $p$  & $-\i q$ \\ 
	&& $p$ & $sr^4$    & $q$ & $-\i p$ \\  
    && $q$ & $sr^6$ & $-u$ & $-\i v$\\ \cline{2-6}
 &  {\multirow{4}{*}{$V_{40}$}}
     & $u$ & $sr$    &$v$  & $-u$ \\ 
    && $v$ & $sr^2$    & $p$ & $-q$ \\ 
	&& $p$ & $sr^4$    & $q$ & $-p$ \\  
    && $q$ & $sr^6$ & $u$ & $-v$ \\ \cline{2-6}
 &  {\multirow{4}{*}{$V_{41}$}}
     & $u$ & $sr$    & $v$ & $-\i u$ \\ 
    && $v$ & $sr^2$    &$p$  & $\i q$ \\ 
	&& $p$ & $sr^4$    & $q$ & $\i p$ \\  
    && $q$ & $sr^6$ & $-u$ & $\i v$ \\ \hline\hline
{\multirow{16}{*}{$[sr]$}} &
      {\multirow{4}{*}{$V_{42}$}}
     & $u$ & $sr$   & $v$  & $q$ \\ 
    && $v$ & $sr^3$    & $p$ & $p$ \\ 
	&& $p$ & $sr^5$    & $q$ & $v$ \\  
    && $q$ & $sr^7$ & $u$ & $u$ \\ \cline{2-6}
 &  {\multirow{4}{*}{$V_{43}$}}
     & $u$ & $sr$   & $v$ & $-\i q$ \\ 
    && $v$ & $sr^3$    & $p$ & $-\i p$ \\ 
	&& $p$ & $sr^5$    & $q$ & $-\i v$ \\  
    && $q$ & $sr^7$ & $-u$ & $-\i u$ \\ \cline{2-6}
 &  {\multirow{4}{*}{$V_{44}$}}
     & $u$ & $sr$   & $v$ & $-q$ \\ 
    && $v$ & $sr^3$    & $p$ & $-p$ \\ 
	&& $p$ & $sr^5$    & $q$ & $-v$ \\  
    && $q$ & $sr^7$ & $-u$ & $-u$ \\ \cline{2-6}
 &  {\multirow{4}{*}{$V_{45}$}}
     & $u$ & $sr$   & $v$ & $\i q$ \\ 
    && $v$ & $sr^3$    & $p$ & $\i p$ \\ 
	&& $p$ & $sr^5$    & $q$ & $\i v$ \\  
    && $q$ & $sr^7$ & $-u$ & $\i u $ \\ \hline
\end{tabular}}
\end{center}
\end{minipage}
\label{irrepsD8}
\end{table}

In Table \ref{tensorSquareD8} we give the decomposition of $V_a\otimes V_b$, $a,b\in\{32,37,45\}$.
The tensor products have been suppressed in the generators, and the generators are listed in the same order
as they appear in Table \ref{irrepsD8}.

\begin{table}[H]
\caption{Decomposition of $(V_{32} \oplus V_{37} \oplus V_{45})^{\otimes 2}$}
\begin{minipage}[t]{.48\textwidth}
\strut\vspace*{-\baselineskip}\newline 
\begin{center}
\begin{tabular}{|c|c|c|}
\hline
Tensor & $V_i$ & Generators \\ \hline\hline
\multirow{4}{*}{$V_{32} \otimes V_{32}$}
  & $V_0$                     & $y_1y_2 + y_2y_1$ \\ \cline{2-3}
  & $V_1$                     & $y_1y_2 - y_2y_1$ \\ \cline{2-3}
  & \multirow{2}{*}{$V_{26}$} & $y_2^2$ \\
  &                           & $y_1^2$ \\ \hline
\multirow{4}{*}{$V_{37} \otimes V_{37}$}
  & $V_0$                     & $z_1z_2 - z_2z_1$ \\ \cline{2-3}
  & $V_1$                     & $z_1z_2 + z_2z_1$ \\ \cline{2-3}
  & \multirow{2}{*}{$V_{24}$} & $z_2^2$ \\
  &                           & $z_1^2$ \\ \hline
\multirow{4}{*}{$V_{32} \otimes V_{37}$}
  & \multirow{2}{*}{$V_{5}$}  & $y_2z_1$ \\
  &                           & $y_1z_2$ \\ \cline{2-3}
  & \multirow{2}{*}{$V_{29}$} & $y_1z_1$ \\
  &                           & $y_2z_2$ \\ \hline
\multirow{4}{*}{$V_{37} \otimes V_{32}$}
  & \multirow{2}{*}{$V_{5}$}  & $z_1y_2$ \\ 
  &                           & $z_2y_1$ \\ \cline{2-3}
  & \multirow{2}{*}{$V_{29}$} & $z_1y_1$ \\ 
  &                           & $z_2y_2$ \\ \hline
\multirow{8}{*}{$V_{32} \otimes V_{45}$}
  & \multirow{4}{*}{$V_{39}$} & $y_1x_2 + y_2x_3$          \\ 
  &                           & $\i y_1x_3 - \i y_2x_4$    \\ 
  &                           & $-y_1x_4 + y_2x_1$         \\ 
  &                           & $\i y_1x_1 - \i y_2x_2$    \\ \cline{2-3}
  & \multirow{4}{*}{$V_{41}$} & $y_1x_2 - y_2x_3$          \\
  &                           & $\i y_1x_3 + \i y_2x_4$    \\
  &                           & $-y_1x_4 - y_2x_1$         \\
  &                           & $\i y_1x_1 + \i y_2x_2$    \\ \hline
\end{tabular}
\end{center}
\end{minipage}
\begin{minipage}[t]{.48\textwidth}
\strut\vspace*{-\baselineskip}\newline 
\begin{center}
\begin{tabular}{|c|c|c|}
\hline
Tensor & $V_i$ & Generators \\ \hline\hline
\multirow{8}{*}{$V_{45} \otimes V_{32}$}
  & \multirow{4}{*}{$V_{39}$} & $x_2y_2 + x_3y_1$          \\ 
  &                           & $-\i x_3y_2 + \i x_4y_1$   \\ 
  &                           & $-x_4y_2 + x_1y_1$         \\ 
  &                           & $-\i x_1y_2 + \i x_2y_1$   \\ \cline{2-3}
  & \multirow{4}{*}{$V_{41}$} & $x_2y_2 - x_3y_1$          \\ 
  &                           & $-\i x_3y_2 - \i x_4y_1$   \\ 
  &                           & $-x_4y_2 - x_1y_1$         \\
  &                           & $-\i x_1y_2 - \i x_2y_1$   \\ \hline
\multirow{8}{*}{$V_{45} \otimes V_{37}$}
  & \multirow{4}{*}{$V_{38}$} & $x_2z_2- \omega^2x_3z_1$          \\ 
  &                           & $\omega x_3z_2-\omega x_4z_1$    \\ 
  &                           & $\omega^2x_4z_2+x_1z_1$           \\ 
  &                           & $-\omega^3 x_1z_2-\omega^3x_2z_1$ \\ \cline{2-3}
  & \multirow{4}{*}{$V_{40}$} & $x_2z_2+ \omega^2x_3z_1$          \\
  &                           & $\omega x_3z_2+\omega x_4z_1$     \\
  &                           & $\omega^2x_4z_2-x_1z_1$           \\
  &                           & $-\omega^3 x_1z_2+\omega^3x_2z_1$ \\ \hline
\multirow{8}{*}{$V_{37} \otimes V_{45}$}
  & \multirow{4}{*}{$V_{38}$} & $z_1x_2+ \omega^2z_2x_3$          \\ 
  &                           & $-\omega^3z_1x_3+\omega^3 z_2x_4$ \\
  &                           & $-\omega^2z_1x_4+z_2x_1$          \\ 
  &                           & $\omega z_1x_1+\omega z_2x_2$     \\ \cline{2-3}
  & \multirow{4}{*}{$V_{40}$} & $z_1x_2- \omega^2z_2x_3$          \\ 
  &                           & $-\omega^3z_1x_3-\omega^3 z_2x_4$ \\ 
  &                           & $-\omega^2z_1x_4-z_2x_1$          \\ 
  &                           & $\omega z_1x_1-\omega z_2x_2$     \\ \hline
\end{tabular}
\end{center}
\end{minipage}
\begin{center}
\begin{tabular}{|c|c|c|}
\hline
Tensor & $V_i$ & Generators \\ \hline\hline
\multirow{16}{*}{$V_{45} \otimes V_{45}$}
  & $V_0$                     & $x_1x_3 + x_2x_4-x_3x_1-x_4x_2$ \\ \cline{2-3}
  & $V_3$                     & $x_1x_3 - x_2x_4-x_3x_1+x_4x_2$ \\ \cline{2-3}
  & \multirow{2}{*}{$V_4$}    & $x_1x_3 +\omega^6 x_2x_4-\omega^4x_3x_1-\omega^2x_4x_2$  \\ 
  &                           & $-x_4x_2 -\omega^6 x_3x_1+\omega^4x_2x_4+\omega^2x_1x_3$ \\ \cline{2-3}
  & $V_8$                     & $x_1^2 + x_2^2 +x_3^2 + x_4^2$ \\ \cline{2-3}
  & $V_9$                     & $x_1^2 - x_2^2 +x_3^2 - x_4^2$ \\ \cline{2-3}
  & \multirow{2}{*}{$V_{11}$} & $x_1^2 +\omega^6 x_2^2+\omega^4x_3^2+\omega^2x_4^2$   \\
  &                           & $-x_4^2-\omega^6 x_3^2 - \omega^4x_2^2-\omega^2x_1^2$ \\ \cline{2-3}
  & \multirow{2}{*}{$V_{22}$} & $x_1x_4 - x_2x_1 - x_3x_2 - x_4x_3$ \\ 
  &                           & $x_2x_3 + x_1x_2 - x_4x_1 + x_3x_4$ \\ \cline{2-3}
  & \multirow{2}{*}{$V_{24}$} & $x_1x_4 -\omega^6 x_2x_1 - \omega^4x_3x_2 - \omega^2x_4x_3$  \\ 
  &                           & $-x_4x_1 + \omega^6x_3x_4 + \omega^4x_2x_3 + \omega^2x_1x_2$ \\ \cline{2-3}
  & \multirow{2}{*}{$V_{26}$} & $x_1x_4 + x_2x_1 - x_3x_2 + x_4x_3$ \\ 
  &                           & $x_2x_3 - x_1x_2 - x_4x_1 - x_3x_4$ \\ \cline{2-3}
  & \multirow{2}{*}{$V_{28}$} & $x_1x_4 +\omega^6 x_2x_1 - \omega^4x_3x_2 + \omega^2x_4x_3$ \\ 
  &                           & $-x_4x_1 - \omega^6x_3x_4 + \omega^4x_2x_3 - \omega^2x_1x_2$ \\ \hline
\end{tabular}
\end{center}
\label{tensorSquareD8}
\end{table}

We need the following lemma, whose proof is a computation using Table 
\ref{irrepsD8}.

\begin{lemma} \label{DQ8decomps}
One has the following direct sum decompositions:
\begin{eqnarray*}
V_{1} \otimes V_{45}  & \cong & V_{43}, \\
V_{5} \otimes V_{45}  & \cong & V_{42} \oplus V_{44},\\
V_{39} \otimes V_{45} & \cong & V_{14} \oplus V_{16} \oplus V_{18} \oplus V_{20} \oplus
                                V_{30} \oplus V_{32} \oplus V_{34} \oplus V_{36}, \\
V_{40} \otimes V_{45} & \cong & V_{15} \oplus V_{17} \oplus V_{19} \oplus V_{21} \oplus
                                V_{31} \oplus V_{33} \oplus V_{35} \oplus V_{37}, \\                            
V_{42} \otimes V_{45} & \cong & V_{5} \oplus V_{6} \oplus V_{12} \oplus V_{13} \oplus
                                V_{23} \oplus V_{25} \oplus V_{27} \oplus V_{29}. \\                            
\end{eqnarray*}
\end{lemma}

\begin{proposition}\label{inner-faithful-dq8}
    The representation $V=V_{32}\oplus V_{37}\oplus V_{45}$ is an inner faithful
    representation of $D(\mathfrak{D}_{8})$.
\end{proposition}
\begin{proof}
The simple modules 
\begin{gather*} V_0,V_1,V_3,V_4, V_5,V_8,V_9,V_{11},V_{22},V_{24},V_{26}, V_{28},V_{29}, V_{38}, V_{39}, V_{40},V_{41}
\end{gather*}
all appear in the decomposition of $V\otimes V$. The simple modules
\begin{gather*} V_{14},V_{15},V_{16},V_{17}, V_{18},V_{19}, V_{20}, V_{21}, V_{30},V_{31}, V_{33},V_{34}, V_{35}, V_{36}, V_{42}, V_{43}, V_{44}
\end{gather*}
all appear in the decomposition of $V_{39}\otimes V_{45}$, $V_{40}\otimes V_{45}$,
$V_{1}\otimes V_{45}$, and $V_{5}\otimes V_{45}$, which appear in the decomposition of $V\otimes V \otimes V$. The remaining simple modules
\begin{gather*} V_{2}, V_{6},V_{7}, V_{10},V_{12},V_{13},V_{23}, V_{25},V_{27}
\end{gather*}
all appear in the decomposition of $V_{42}\otimes V_{45}$ and $V_{43}\otimes V_{45}$,
which appear in the decomposition of $V\otimes V\otimes V \otimes V$.
\end{proof}

We now construct an algebra $B$ on which $D(\mathfrak{D}_8)$ acts
with degree one component equal to $V$.  As in the decomposition tables above, we use 
$x_1,\dots,x_4$ for the basis of $V_{45}$, $y_1,y_2$ for the basis of 
$V_{32}$ and $z_1,z_2$ for the basis of $V_{37}$.  
Let $\alpha, \beta, \gamma \in \{\pm 1\}$, let
$\constb,\constc,\constd,\conste \in \Bbbk^*$,
and let $\omega = e^{\frac{\pi i}{4}}$.   Consider the following set of relations taken
from the decomposition of $V \otimes V$ appearing in Table \ref{relationsQ8}.
The ``Rep" column indicates which representation corresponds to the relations
appearing in the ``Relations" column.  If multiple subscripts are given, then
the representation depends on the choice of the parameter that appears in the
relation column, with the first (respectively, second) integer corresponding to the
parameter $1$ (respectively, $-1$).
\begin{table}[H]
\caption{Relations for an algebra with action of $D(\mathfrak{D}_{8})$}
\begin{minipage}[t]{.45\textwidth}
\strut\vspace*{-\baselineskip}\newline
\scalebox{0.75}{
\begin{tabular}{|c|c|} 
\hline
Rep & Relations \\ \hline
\multirow{2}{*}{$V_{4}$}       & $x_1x_3 + x_3x_1$ \\ \cline{2-2}
                               & $x_2x_4 + x_4x_2$ \\ \hline
\multirow{2}{*}{$V_{22,26}$}   & $x_1x_4 - \alpha x_2x_1 - x_3x_2 - \alpha x_4x_3$ \\ \cline{2-2}
                               & $x_2x_3 + \alpha x_1x_2 - x_4x_1 + \alpha x_3x_4$ \\ \hline
\multirow{2}{*}{$V_{24,28}$}   & $x_1x_4 + \alpha \i x_2 x_1 + x_3x_2 - \alpha\i x_4x_3$ \\ \cline{2-2}
                               & $-x_4x_1 - \alpha \i x_3x_4 - x_2x_3 + \alpha \i x_1x_2$ \\ \hline
$V_{1,0}$ & $y_1y_2 - \beta y_2y_1$  \\ \hline
$V_{0,1}$ & $z_1z_2 - \gamma z_2z_1$ \\ \hline
\multirow{4}{*}{$V_{39}$}      & $(y_1x_2 + y_2x_3) - \constd(x_2y_2 + x_3y_1)$ \\ \cline{2-2}
                               & $(\i y_1x_3 - \i y_2x_4) - \constd(-\i x_3y_2 + \i x_4y_1)$ \\ \cline{2-2}
                               & $(-y_1x_4 + y_2x_1) - \constd(-x_4y_2 + x_1y_1)$ \\ \cline{2-2}
                               & $(\i y_1x_1 - \i y_2x_2) - \constd(-\i x_1y_2 + \i x_2y_1)$ \\ \hline
\multirow{2}{*}{$V_5$}         & $z_1y_2 - \constb y_2z_1$ \\ \cline{2-2}
                               & $z_2y_1 - \constb y_1z_2$ \\ \hline
\end{tabular}}
\end{minipage}
\begin{minipage}[t]{.45\textwidth}
\strut\vspace*{-\baselineskip}\newline
\scalebox{0.75}{\begin{tabular}{|c|c|}
\hline
Rep & Relations \\ \hline
\multirow{4}{*}{$V_{38}$} & $(z_1x_2 + \omega^2 z_2x_3) - \conste(x_2z_2 - \omega^2 x_3z_1)$ \\ \cline{2-2}
                          & $(-\omega^3 z_1x_3 + \omega^3 z_2x_4) - \conste(\omega x_3z_2 - \omega x_4z_1)$ \\ \cline{2-2}
                          & $(-\omega^2z_1x_4 + z_2x_1) - \conste(\omega^2 x_4z_2 + x_1z_1)$ \\ \cline{2-2}
                          & $(\omega z_1x_1 + \omega z_2x_2) - \conste(-\omega^3x_1z_2 - \omega^3x_2z_1)$ \\ \hline
\multirow{4}{*}{$V_{40}$} & $(z_1x_2 - \omega^2 z_2x_3) - \conste(x_2z_2 + \omega^2 x_3z_1)$ \\ \cline{2-2}
                          & $(-\omega^3 z_1x_3 - \omega^3 z_2x_4) - \conste(\omega x_3z_2 + \omega x_4z_1)$ \\ \cline{2-2}
                          & $(-\omega^2z_1x_4 - z_2x_1) - \conste(\omega^2 x_4z_2 - x_1z_1)$ \\ \cline{2-2}
                          & $(\omega z_1x_1 - \omega z_2x_2) - \conste(-\omega^3x_1z_2 + \omega^3x_2z_1)$ \\ \hline
\multirow{4}{*}{$V_{41}$} & $(y_1x_2 - y_2x_3) - \constd(x_2y_2 - x_3y_1)$ \\ \cline{2-2}
                          & $(\i y_1x_3 + \i y_2x_4) - \constd(-\i x_3y_2 - \i x_4y_1)$ \\ \cline{2-2}
                          & $(-y_1x_4 - y_2x_1) - \constd(-x_4y_2 - x_1y_1)$ \\ \cline{2-2}
                          & $(\i y_1x_1 + \i y_2x_2) - \constd(-\i x_1y_2 - \i x_2y_1)$ \\ \hline
\multirow{2}{*}{$V_{29}$} & $z_1y_1 - \constc y_1z_1$ \\ \cline{2-2}
                          & $z_2y_2 - \constc y_2z_2$ \\ \hline
\end{tabular}}
\end{minipage}
\label{relationsQ8}
\end{table}

One sees that each of these relations correspond to a linear combination of generators
of irreducible representations which appear in $V \otimes V$.  Any linear combinations
used follow the method in Lemma \ref{MinsideMM} to produce a copy of a subrepresentation
$W$ inside of $W \oplus W$.  For example, this is used on all of the relations that
use the parameters $\constb,\constc,\constd$ and $\conste$.  Specifically, since $V_{32} \otimes V_{45} \cong
V_{45} \otimes V_{32} \cong V_{39} \oplus V_{41}$, there are two copies of $V_{39}$
and $V_{41}$ in $V \otimes V$.  Our table identifies an isomorphism between these two
copies of each of $V_{39}$ and $V_{41}$, and we use the $\constd$ parameter to choose a particular
copy of $V_{39}$ inside of $V_{39} \oplus V_{39}$ and similarly for $V_{41}$.


\begin{theorem}\label{i-double-ore}
Let $V = V_{32} \oplus V_{37} \oplus V_{45}$ be the representation of
$D(\mathfrak{D}_8)$ given by the table on the left below:

\begin{minipage}[t]{.45\textwidth}
\strut\vspace*{-\baselineskip}\newline 
\begin{center}
\scalebox{.75}{
\begin{tabular}{|c|c|c|c|c|} \hline
$V_i$ & Gen. & Gr. & $r$ & $s$ \\ \hline\hline
{\multirow{2}{*}{$V_{32}$}}
      & $y_1$ & $r^3$   & $\omega^2y_1$ & $y_2$ \\ 
      & $y_2$ & $r^5$   & $\omega^6y_2$ & $y_1$ \\ \hline
{\multirow{2}{*}{$V_{37}$}}
      & $z_1$ & $r^3$   & $\omega^7z_1$ & $z_2$ \\ 
      & $z_2$ & $r^5$   & $\omega  z_2$ & $-z_1$ \\ \hline
{\multirow{4}{*}{$V_{45}$}}
     & $x_1$ & $sr$      & $x_2$  & $\i x_4$ \\ 
     & $x_2$ & $sr^3$    & $x_3$  & $\i x_3$ \\ 
     & $x_3$ & $sr^5$    & $x_4$  & $\i x_2$ \\  
     & $x_4$ & $sr^7$    & $-x_1$ & $\i x_1$ \\ \hline
\end{tabular}}
\end{center}
\end{minipage}
\begin{minipage}[t]{.45\textwidth}
\strut\vspace*{-\baselineskip}\newline
\def\arraystretch{1.2}
\begin{center}
\scalebox{.75}{
\begin{tabular}{|cc|} 
\hline
\multicolumn{2}{|c|}{Relations} \\ \hline \hline
    $y_1y_2 - \beta y_2y_1$      & $z_1z_2 - \gamma z_2z_1$     \\
    $y_1x_1 + \constd x_1y_2$    & $y_2x_1 - \constd x_1y_1$    \\ 
    $y_1x_2 - \constd x_2y_2$    & $y_2x_2 + \constd x_2y_1$    \\
    $y_1x_3 + \constd x_3y_2$    & $y_2x_3 - \constd x_3y_1$    \\
    $y_1x_4 - \constd x_4y_2$    & $y_2x_4 + \constd x_4y_1$    \\
    $z_1x_1 + \conste \i x_1z_2$ & $z_2x_1 - \conste x_1z_1$    \\
    $z_1x_2 - \conste x_2z_2$    & $z_2x_2 + \conste \i x_2z_1$ \\
    $z_1x_3 - \conste \i x_3z_2$ & $z_2x_3 + \conste x_3z_1$    \\
    $z_1x_4 + \conste x_4z_2$    & $z_2x_4 - \conste \i x_4z_1$ \\
    $z_1y_1 - \constc y_1z_1$    & $z_2y_1 - \constb y_1z_2$    \\
    $z_1y_2 - \constb y_2z_1$    & $z_2y_2 - \constc y_2z_2$    \\
    $x_1x_3 + x_3x_1$            & $x_2x_4 + x_4x_2$            \\
\multicolumn{2}{|c|}{$x_2x_1 - \frac{\alpha}{\sqrt{2}}(\omega^3 x_3x_2 + \omega x_1x_4)$} \\
\multicolumn{2}{|c|}{$x_4x_1 - \frac{\alpha}{\sqrt{2}}(\omega x_1x_2 + \omega^7 x_3x_4)$} \\
\multicolumn{2}{|c|}{$x_2x_3 - \frac{\alpha}{\sqrt{2}}(\omega^3 x_1x_2 + \omega^5 x_3x_4)$} \\
\multicolumn{2}{|c|}{$x_4x_3 - \frac{\alpha}{\sqrt{2}}(\omega^5 x_3x_2 + \omega^7 x_1x_4)$} \\ \hline
\end{tabular}}
\end{center}
\def\arraystretch{1}
\end{minipage}

Let $B$ be the quotient algebra of $\Bbbk\langle x_1, x_2, x_3, x_4, y_1, y_2, z_1, z_2 \rangle$ given by the 
relations from the table on the right, or equivalently, by those given in Table \ref{relationsQ8}.
Then for any nonzero $\constb,\constc,\constd,\conste \in \Bbbk$,
and for any $\alpha,\beta,\gamma \in \{\pm 1\}$,
the algebra $B$ is an Artin-Schelter regular algebra of dimension eight upon which $D(\mathfrak{D}_8)$
acts inner faithfully.  More precisely, this algebra may be expressed as an iterated double Ore extension of a 
$(-1)$-skew polynomial ring in two variables.
\end{theorem}

\begin{proof}

We first note that the relations given in the theorem statement are linear combinations
of those appearing in Table \ref{relationsQ8}, as was done in Theorem \ref{double-ore}.
It follows that the span of these relations is a $D(\mathfrak{D}_8)$-module, and hence
$D(\mathfrak{D}_8)$ acts on $B$.

We now describe $B$ as an iterated double Ore extension.
First, note that the last six equations listed define a double Ore extension
of a $(-1)$-skew polynomial ring in two variables.
Indeed, using Notation  \ref{doubleOreNotation}, 
one may take $C_1 = \Bbbk_{-1}[x_1,x_3]$ and let $\sigma^{(1)}$ be given by:

\begin{eqnarray*}
\sigma^{(1)}(x_1) =
\frac{\alpha}{\sqrt{2}}
\begin{pmatrix} 
\omega^3 x_3 & \omega x_1 \\
\omega x_1 & \omega^7 x_3
\end{pmatrix}
\qquad
\sigma^{(1)}(x_3) =
\frac{\alpha}{\sqrt{2}}
\begin{pmatrix} 
\omega^3 x_1 & \omega^5 x_3 \\
\omega^5 x_3 & \omega^7 x_1
\end{pmatrix}.
\end{eqnarray*}
One can check directly that $\sigma^{(1)}$ is an algebra map, and also that
$\sigma^{(1)}$ is invertible in the sense
of \cite[Definition 1.8]{ZZ}.

Let $C_2$ be this double Ore extension.  Next, we describe the algebra defined by those
relations involving the $x$s and $y$s as a double Ore extension of $C_2$.  Indeed,
we may use the matrices $\sigma^{(2)}$ to define this double Ore extension:
\begin{eqnarray*}
\sigma^{(2)}(x_1) =
\constd\begin{pmatrix} 
0   & -x_1 \\
x_1 & 0
\end{pmatrix}
\qquad
\sigma^{(2)}(x_2) =
\constd\begin{pmatrix} 
0   & x_2 \\
-x_2 & 0
\end{pmatrix}
\\
\sigma^{(2)}(x_3) =
\constd\begin{pmatrix} 
0   & -x_3 \\
x_3 & 0
\end{pmatrix}
\qquad
\sigma^{(2)}(x_4) =
\constd\begin{pmatrix} 
0   & x_4 \\
-x_4 & 0
\end{pmatrix}
.
\end{eqnarray*}
Here, we may appeal to Remark \ref{diagDOData} to see that $\sigma^{(2)}$
is an algebra map, and one may verify directly that
$\sigma^{(2)}$ is invertible.
Let $C_3$ be the double Ore extension of $C_2$ 
obtained by adding $y_1$ and $y_2$
via the map $\sigma^{(2)}$.  Finally, we adjoin 
$z_1$ and $z_2$ using
a double Ore extension of $C_3$ according to the map $\sigma^{(3)}$ given below:
\begin{eqnarray*}
\sigma^{(3)}(x_1) =
\conste\begin{pmatrix} 
0   & -\i x_1 \\
x_1 & 0
\end{pmatrix}
\qquad
\sigma^{(3)}(x_2) =
\conste\begin{pmatrix} 
0   & x_2 \\
-\i x_2 & 0
\end{pmatrix}
\\
\sigma^{(3)}(x_3) =
\conste\begin{pmatrix} 
0   & \i x_3 \\
-x_3 & 0
\end{pmatrix}
\qquad
\sigma^{(3)}(x_4) =
\conste\begin{pmatrix} 
0   & -x_4 \\
\i x_4 & 0
\end{pmatrix}
\\
\sigma^{(3)}(y_1) =
\begin{pmatrix} 
\constc y_1   & 0 \\
0 & \constb y_1
\end{pmatrix}
\qquad
\sigma^{(3)}(y_2) =
\begin{pmatrix} 
\constb y_2 & 0 \\
0              & \constc y_2
\end{pmatrix}
.
\end{eqnarray*}
One may check once again that $\sigma^{(3)}$ is 
invertible.  The double Ore extension created from this data
is the algebra $B$ given in the theorem.
\end{proof}

Recall that the actions of $r$ and $s$ on $V = \Span\{x_1,x_2,x_3,x_4,y_1,y_2,z_1,z_2\}$ are given by the matrices
\begin{equation}\label{r-s-action2}
\begin{pmatrix}
0  & 1 & 0 & 0 & 0  & 0   & 0        & 0 \\
0  & 0 & 1 & 0 & 0  & 0   & 0        & 0 \\
0  & 0 & 0 & 1 & 0  & 0   & 0        & 0 \\
-1 & 0 & 0 & 0 & 0  & 0   & 0        & 0 \\
0  & 0 & 0 & 0 & \i & 0   & 0        & 0 \\
0  & 0 & 0 & 0 & 0  & -\i & 0        & 0 \\
0  & 0 & 0 & 0 & 0  & 0   & \omega^7 & 0 \\
0  & 0 & 0 & 0 & 0  & 0   & 0        & \omega \\
\end{pmatrix}
\text{ and }
\begin{pmatrix}
0  & 0  & 0  & \i & 0 & 0 & 0  & 0 \\
0  & 0  & \i & 0  & 0 & 0 & 0  & 0 \\
0  & \i & 0  & 0  & 0 & 0 & 0  & 0 \\
\i & 0  & 0  & 0  & 0 & 0 & 0  & 0 \\
0  & 0  & 0  & 0  & 0 & 1 & 0  & 0 \\
0  & 0  & 0  & 0  & 1 & 0 & 0  & 0 \\
0  & 0  & 0  & 0  & 0 & 0 & 0  & 1 \\
0  & 0  & 0  & 0  & 0 & 0 & -1 & 0 \\
\end{pmatrix},
\end{equation}
respectively.  We seek to compute the homological determinant of the action
of these elements on $B$.  To this end, a computation shows that the twisted superpotential begins as:
$$\mathbf{w}_B = x_1x_3x_2x_4y_1y_2z_1z_2 - x_4x_2x_1x_3y_1y_2z_1z_2 - \beta\gamma x_4x_2x_3x_1y_2y_1z_2z_1 + \cdots.$$
One has the following proposition, whose proof follows from a computation similar to the one
preceding Proposition \ref{prop:hdetQ4}:
\begin{proposition} \label{prop:hdetQ8}
Let $D(\mathfrak{D}_8)$ act on the algebra $B$ given by Theorem \ref{i-double-ore}.
Then the group grade component of the homological determinant
of this action is trivial, and one has that
$$\hdet_B(r) = 1 \qquad\text{and}\qquad \hdet_B(s) = \beta\gamma.$$
In particular, this action has trivial homological determinant if and only
if $\beta = \gamma$.
\end{proposition}
One may also check that the Nakayama automorphism of this algebra is:
$$\mu =
\begin{pmatrix}
\beta\gamma \constd^2\conste^2 I_4 & 0 & 0 \\
 0 & \beta\constb^{-1}\constc^{-1}\constd^4 I_2 & 0 \\
 0 & 0 & -\gamma\constb^{-1}\constc^{-1}\conste^{-4} I_2
\end{pmatrix}$$




\end{section}

\begin{section}{The symmetric group \texorpdfstring{$S_3$}{S3}}
\label{S3}
In this section, we show that any connected graded algebra on which
$D(S_3)$ acts inner faithfully must have at least three generators.
Further, we show that there is no AS regular algebra (indeed, no 
domain) $B$ with three generators and quadratic relations on which
$D(S_3)$ acts inner faithfully.  


The Hopf algebra $D(S_3)$ has 8 simple modules,
which we denote $V_i$ for $i = 0, \dots, 7$.  Recall
that each simple module of $D(S_3)$ is given as in
Notation \ref{not:modulesOverDouble}.
The conjugacy classes (with choice of representative $a$)  centralizers of 
$a$, and irreducible characters are given in Table 
\ref{DS3simples}, which uses the following notation.

We use $\zeta$ to denote a primitive third root of 
unity, $\psi_0$ and $\psi_1$ denote the trivial 
representation and sign representation, respectively,
and $\chi$ has $S_3$-action given by $u.r = \zeta u$ and $v.r = \zeta^2 v$,
$u.s = v$ and $v.s = u$.  Finally, the representations of the (cyclic)
centralizers are given by $\alpha_i: t.r=\zeta^it$ for $i=0,1,2$ and
$\beta_j: t.s=(-1)^jt$ for $j=0,1$. 

\begin{table}[H]
\caption{The simple $D(S_3)$ modules}
\begin{center}
\scalebox{.8}{
 \begin{tabular}{|c|c|c|c|}
\hline
  Conj class & $C_G(a)$ & Rep $C_G(a)$& transveral\\
\hline
\hline
$\{a=e\}$ & $S_{3}$ & $\psi_0, \psi_1, \chi$ & $\{e\}$\\
\hline
 &  & $V_0, V_1, V_2$& \\
\hline
$\{a=r,r^2\}$ & $\mathbb{Z}_{3} = \langle r \rangle$  & $\alpha_0, \alpha_1, \alpha_2$& $\{e, s\}$\\
\hline
 &  & $V_3, V_4, V_5$ & \\
\hline
$\{a=s,sr,sr^2\}$ & $\mathbb{Z}_2 = \langle s \rangle $ & $\beta_0, \beta_1$ & $\{e, r, r^2\}$ \\
\hline
 &  & $V_{6}, V_{7}$ & \\
\hline
\end{tabular}}
\end{center}
\label{DS3simples}
\end{table}

Table \ref{dihedral-n-2 table:} gives a summary of 
the action of the generators of $D(S_3)$ on its 
simple modules $V_i, i=0,1,\ldots,7.$
\begin{table}[H]
\caption{Explicit action of $D(S_3)$ on its simple modules}
\begin{center}
\scalebox{.8}{
\begin{tabular}{|c|c|c|c|c|c|c|} \hline
Conj Class & $C_G(a)$ & $V_i$ & Generator & Grade & $r$ & $s$ \\ \hline\hline
{\multirow{4}{*}{$[a]$}} & {\multirow{4}{*}{$G$}} &
     $V_0$ & $u$ & $e$ & $u$ & $u$ \\ \cline{3-7}
 &&  $V_1$ & $u$ & $e$ & $u$ & $-u$ \\ \cline{3-7}
 &&  {\multirow{2}{*}{$V_2$}} 
       & $u$ & $e$ & $\zeta u$ & $v$ \\
     &&& $v$ & $e$ & $\zeta^2 v$ & $u$ \\ \hline\hline
{\multirow{6}{*}{$[r]$}} & {\multirow{6}{*}{$\mathbb{Z}_3$}} &
   {\multirow{2}{*}{$V_3$}} 
       & $u$ & $r$ & $u$ & $v$ \\
     &&& $v$ & $r^2$ & $v$ & $u$ \\ \cline{3-7}
  && {\multirow{2}{*}{$V_4$}} 
       & $u$ & $r$ & $\zeta u$ & $v$ \\
     &&& $v$ & $r^2$ & $\zeta^2 v$ & $u$ \\ \cline{3-7} 
 && {\multirow{2}{*}{$V_5$}} 
       & $u$ & $r$ & $\zeta^2 u$ & $v$ \\
     &&& $v$ & $r^2$ & $\zeta v$ & $u$ \\ \hline\hline
{\multirow{6}{*}{$[s]$}} & {\multirow{6}{*}{$\mathbb{Z}_2$}} &
   {\multirow{3}{*}{$V_6$}} 
       & $u$ & $s$ & $p$ & $u$ \\
     &&& $v$ & $sr$ & $u$ & $p$ \\  
    &&& $p$ & $sr^2$ & $v$ & $v$ \\ \cline{3-7} 
 && {\multirow{3}{*}{$V_7$}} 
       & $u$ & $s$ & $p$ & $-u$ \\
     &&& $v$ & $sr$ & $u$ & $-p$ \\ 
       &&& $p$ & $sr^2$ & $v$ & $-v$ \\ \hline\hline
\end{tabular}}
\end{center}
\label{dihedral-n-2 table:}
\end{table}
 To show that $V_6$ is an inner faithful $D(S_3)$-module, we 
 will decompose $V_6\otimes V_6$, $V_6\otimes V_2$ and 
 $V_2\otimes V_2$. We let $V_{6} = \Span_k\{x_1,x_2,x_3\}$ and 
 $V_{2} = \Span_k\{y_1,y_2\}$.
The explicit decomposition of $V_6 \otimes V_6$ is given in Table \ref{V_{6}X V_{6}}.
\begin{table}[H]
\caption{The decomposition of $V_6 \otimes V_6$}
\begin{center}
\begin{tabular}{|c|c|c|c|c|c|c|}
\hline
Tensor & Simple & grade & gens  &$r$-action & $s$-action \\ \hline\hline
\multirow{9}{*}{$V_{6} \otimes V_{6}$}
& \multirow{1}{*}{$V_{0}$} & $e$ & $u= x_1^2+x_2^2+x_3^2$ & $u$  & $u$ \\ \cline{2-6}
& \multirow{2}{*}{$V_{2}$} & $e$ & $u= x_1^2+\zeta x_2^2+\zeta^2 x_3^2$ & $\zeta u$  & $v$ \\ \cline{3-6}
&  & $e$  & $v=  x_1^2+\zeta^2x_2^2+\zeta x_3^2$ & $\zeta^2v$  & $u$ \\ \cline{2-6}
& \multirow{2}{*}{$V_{3}$} 
& $r$       & $u= x_1x_2+x_3x_1+ x_2x_3$ & $u$  & $v$ \\ \cline{3-6}
&  & $r^2$  & $v= x_1x_3+x_2x_1+x_3x_2$ & $v$  & $u$ \\ \cline{2-6}
& \multirow{2}{*}{$V_{4}$} 
& $r$       & $u= x_1x_2+\zeta^2 x_3x_1+ \zeta x_2x_3$ & $\zeta u$  & $v$ \\ \cline{3-6}
&  & $r^2$  & $v= x_1x_3+\zeta^2 x_2x_1+ \zeta x_3x_2$ & $\zeta^2 v$  & $u$ \\ \cline{2-6}
& \multirow{2}{*}{$V_{5}$} 
& $r$       & $u= x_1x_2+\zeta x_3x_1+ \zeta^2 x_2x_3$ & $\zeta^2 u$  & $v$ \\ \cline{3-6}
&  & $r^2$  & $v= x_1x_3+\zeta x_2x_1+ \zeta^2 x_3x_2$ & $\zeta v$  & $u$ \\ \hline
\end{tabular}
\label{V_{6}X V_{6}}
\end{center}
\end{table}
The decompositions of the remaining tensor products
are given in the following lemma, which can be proved by a direct calculation using Table \ref{dihedral-n-2 table:}.
\begin{lemma}
Let $i = 6,7$.  Then one has the following decompositions:
\begin{eqnarray*}
V_i \otimes V_i & \cong & V_0 \oplus V_2 \oplus V_3 \oplus V_4 \oplus V_5, \\
V_i \otimes V_2 & \cong & V_6 \oplus V_7, \\
V_2 \otimes V_2 & \cong & V_0 \oplus V_1 \oplus V_2.
\end{eqnarray*}
\end{lemma}

\begin{proposition}\label{inner-faithful-D3}
A $D(S_3)$-module $V$  is inner faithful if and only if $V_6$ or $V_7$
is a direct summand of $V$.
\end{proposition}
\begin{proof}
Let $V = V_6$ or $V_7$.  Then the simple modules 
\begin{gather*} V_0,V_2,V_3,V_4, V_5
\end{gather*}
all appear in the decomposition of $V\otimes V$.
The simple modules  $V_6$ and $V_7$ appear in the decomposition of
$V\otimes V_{2}$ which appears in the decomposition of $V^{\otimes 3}$. The remaining simple module
$V_1$ appears in the decomposition of $V_2\otimes V_2$ which appears in $V^{\otimes 4}$.
For the converse, one may check that $V_i \otimes V_j$ for $0 \leq i,j \leq 5$ do not
contain $V_6$ or $V_7$.
\end{proof}

\begin{theorem}
Let $B$ be a connected graded algebra generated in degree one with three generators
and three quadratic relations on which $D(S_3)$ acts inner-faithfully.  Then $B$ is not a domain.
\end{theorem}
\begin{proof}
By Proposition \ref{inner-faithful-D3}, one must have $B_1 \cong V_6$ or $V_7$ as a module over
$D(S_3)$.  One may check that $V_6 \otimes V_6 \cong V_7 \otimes V_7$, so without loss of generality
we may assume that $V = V_6$.

By the decomposition of $V_6 \otimes V_6$ above, we have four possible such algebras on which $D(S_3)$ acts.
The first such algebra uses relations from $V_0$ and $V_2$, which gives the algebra
$\Bbbk\langle x,y,z \rangle/\langle x^2,y^2,z^2 \rangle$ which is clearly not a domain.
The second such algebra uses relations from $V_0$ and $V_3$, and is
$\Bbbk\langle x,y,z \rangle/\langle x^2 + y^2 + z^2, xz + yx + zy, xy + zx + yz \rangle$.  One may check that
these three relations form a Gr\"obner basis of the ideal using graded lexicographic order with $x>y>z$, and
using this one can show that the elements $x + y + z$ and $yx - yz + zy - z^2$ are nonzero, but have 
product equal to zero.  The last two algebras are related by choosing a different primitive third
root of unity $\zeta$, so we  give details only when we use the relations coming from $V_0$ and $V_4$.
In this case, one sees that $(\zeta^2x + y + z)(\zeta x + y + z) = 0$.
\end{proof}
\end{section}
\begin{section}{Remarks and Conjectures}\label{questions}

The doubles $D(\mathfrak{D}_{4})$ and $D(\mathfrak{D}_{8})$ studied here were explored after unsuccessful attempts to find AS regular algebras $A$ on which the doubles of the dihedral group of order 8 and the dicyclic group of order 12 act inner faithfully.  In each of the groups $G$ that we considered there were numerous inner faithful $D(G)$-modules, but we were unsuccessful in finding relations for an AS regular algebra $A$ on which $D(G)$ acts inner faithfully.
We conclude with a number of questions.
\begin{question} For $n=2,4$ characterize the inner faithful $D(\mathfrak{D}_{2n})$-modules $V$ for which there an AS regular algebra $A$ with $A_1 = V$ on which $D(\mathfrak{D}_{2n})$ acts inner faithfully. Macaulay2 produced a large number of choices for $V$ having a minimal number of simple summands, and we did not consider in depth all of the possibilities, or the effect of adding more summands to $V$  (which will preserve the inner faithful action property). There were some choices of $V$ where we were unable to find such an algebra $A$.
\end{question}
\begin{question} For which dicyclic groups $\mathfrak{D}_{2n}$ does there exist an AS regular algebra $A$ on which $\mathfrak{D}_{2n}$ acts inner faithfully?  Perhaps $n$ must be even or a power of $2$.
\end{question}
\begin{question} Given an AS regular algebra $A$ that is a double Ore extension, what are the Hopf algebras $H$ that act inner faithfully on it? There is an AS regular  algebra $A$ of dimension 4 which is a double Ore extension on which the dual of the group algebra of the semidihedral group of order 16 acts inner faithfully (see \cite{GKMV}), and Ore extensions have appeared quite often as the AS regular algebras on which a Hopf algebra acts.
\end{question}
\begin{question} For which finite groups $G$ is there an AS regular algebra  $A$ on which $D(G)$ acts inner faithfully?  What properties characterize those representations of $G$?
\end{question}

\begin{question} How are properties of the ring of invariants under the action of $D(G)$ controlled by $G$, or by properties of both $G$ and the dual of $G$?
\end{question}

\begin{question}
The double of a group $D(G)$ is a quasi-triangular Hopf algebra and an almost cocommutative Hopf algebra. The fact that $V_1 \otimes V_2 \cong V_2 \otimes V_1$ was useful in constructing the algebras $A$. Are there properties of a Hopf algebra $H$ (or of its representation ring) (or of doubles $D(H)$) that produce appropriate AS regular algebras? 
\end{question}
There remain many interesting questions about the pairs $(A,H)$, where $A$ is an AS regular algebra supporting an inner faithful action by a Hopf algebra $H$.\\
~\\
~\\
\end{section}
\noindent
{\bf Acknowledgement}:  Some of this work was supported by the National Science Foundation under Grant No. DMS-1928930, while the first author was in residence at the Simons Laufer Mathematical Sciences Institute in Berkeley, California, during the spring semester of 2024.

\end{document}